\pdfoutput=1
\documentclass[10pt]{amsart}
\usepackage{amssymb}

\usepackage[theoremfont]{newtxtext}
\usepackage[scaled=0.95]{roboto}
\usepackage[varbb, varvw, smallerops]{newtxmath}
\usepackage[scaled=0.87]{inconsolata}
\usepackage[cal=cm, calscaled=0.95, frak=euler, frakscaled=0.98]{mathalfa}
\usepackage[T1]{fontenc}
\usepackage{comment}

\usepackage{tikz}
\usepackage{tikz-cd}
\usetikzlibrary{arrows}
\tikzcdset{arrow style=tikz,
diagrams={>={Stealth[round,length=4pt,width=6pt,inset=3pt]}}}
\tikzcdset{every label/.append style = {font = \footnotesize}}
\usepackage{graphicx}
\usepackage{hyperref}
\hypersetup{colorlinks=false}

\title{Derived Complete Complexes at Weakly Proregular Ideals}
\author{Amnon Yekutieli}
\date{3 August 2024. {\em Version}: 4.}
\address{Department of  Mathematics,
Ben Gurion University, Be'er Sheva 84105, Israel.}
\email{\href{mailto:amyekut@gmail.com}{amyekut@gmail.com}}
\urladdr{\url{https://sites.google.com/view/amyekut-math/home}}
\dedicatory{To the memory of my dear friend Sahar Mientakevitch}

\newtheorem{thm}[equation]{Theorem}
\newtheorem{cor}[equation]{Corollary}
\newtheorem{prop}[equation]{Proposition}
\newtheorem{lem}[equation]{Lemma}
\theoremstyle{definition}
\newtheorem{dfn}[equation]{Definition}
\newtheorem{rem}[equation]{Remark}

\newtheorem{conv}[equation]{Convention}

\numberwithin{equation}{section}

\newcommand{\iso}{\xrightarrow{%
\smash{\raisebox{-0.5ex}{\ensuremath{\scriptstyle \simeq  \mspace{2mu}}}}}} 
\newcommand{\xar}{\xrightarrow}

\newcommand{\lto}{\leftarrow}
\newcommand{\sub}{\subseteq}
\newcommand{\opn}{\operatorname}
\newcommand{\cat}[1]{\operatorname{\mathsf{#1}}}

\newcommand{\cd}{\mspace{1.3mu}{\cdotB}\mspace{1.3mu}}

\newcommand{\rmitem}[1]{\item[\textrm{(#1)}]}
\newcommand{\mfrak}[1]{\mathfrak{#1}}

\newcommand{\mrm}[1]{\mathrm{#1}}

\newcommand{\Ga}{\Gamma}
\newcommand{\si}{\sigma}

\newcommand{\de}{\delta}

\newcommand{\al}{\alpha}

\newcommand{\ga}{\gamma}

\newcommand{\La}{\Lambda}
\renewcommand{\a}{\mfrak{a}}
\renewcommand{\b}{\mfrak{b}}
\renewcommand{\c}{\mfrak{c}}

\newcommand{\ba}{\bsym{a}}
\newcommand{\bb}{\bsym{b}}

\newcommand{\bm}{\bsym{m}}

\newcommand{\bn}{\bsym{n}}

\newcommand{\K}{\mathbb{K}}

\newcommand{\Z}{\mathbb{Z}}
\newcommand{\N}{\mathbb{N}}

\newcommand{\tup}[1]{\textrm{#1}}
\newcommand{\bsym}[1]{\boldsymbol{#1}}

\newcommand{\boplus}{\bigoplus\nolimits}
\newcommand{\ot}{\otimes}

\newcommand{\what}[1]{\widehat{#1}}
\newcommand{\wh}[1]{\widehat{#1}}

\renewcommand{\d}{\mathrm{d}}

\newcommand{\lb}{\linebreak}

\newcommand{\msp}[1]{\mspace{#1 mu}}

\begin{document}

\begin{abstract}
Weak proregularity of an ideal in a commutative ring is a subtle
generalization of the noetherian property of the ring.
Weak proregularity is of special importance for the study of derived completion,
and it occurs quite often in non-noetherian rings arising in Hochschild
and prismatic cohomologies.

This paper is about several related topics: adically flat modules,
recognizing derived complete complexes, the structure of the category of derived
complete complexes, and a derived complete Nakayama theorem -- all with
respect to a weakly proregular ideal; and the preservation
of weak proregularity under completion of the ring.
\end{abstract}

\maketitle

\tableofcontents

\setcounter{section}{-1}
\section{Introduction}

The {\em weak proregularity} (WPR) property of an ideal $\a$ in a commutative
ring $A$ was discovered by Grothendieck in 1967 \cite{LC}, without naming it.
A finite sequence $(a_1, \ldots, a_n)$ of elements of $A$ is called a {\em
weakly proregular sequence} if its Koszul cohomology satisfies a rather
complicated asymptotic condition (see Definition \ref{dfn:255} below).
An ideal $\a$ in $A$ is called WPR if it is generated by some WPR sequence.
Grothendieck proved that when $A$ is a noetherian
ring, every ideal in it is WPR. Moreover, he proved that the WPR property of
$\a$ is sufficient for the {\em derived $\a$-torsion} (the algebraic version of
cohomology with supports) to behave like in the noetherian case.

Thirty years later, weak proregularity was brought back into active research
by Alonso, Jeremias and Lipman \cite{AJL}, and they also coined the
name (in the Erratum of their paper). There was much
subsequent work by other authors, including the proof that the {\em
MGM Equivalence} holds for WPR ideals \cite{PSY1}.
The relation between WPR and {\em adic flatness} (a slight
weakening of the flatness condition) was addressed in \cite{Ye4},
and a noncommutative generalization of WPR was studied in \cite{VY}.
Further developments can be found in the papers \cite{Ye6} and \cite{Po2}.
To summarize: These prior results showed that {\em the WPR
property of the ideal $\a$ is a very subtle generalization of the noetherian
property of the ring $A$}.

Several applications of the WPR property were found,
including to {\em Hochschild cohomology}, see \cite{Sh1}, \cite{Sh2} and
\cite{Sh3}; and to {\em perfectoid and prismatic theory}, see \cite{CS},
\cite{BS}, \cite{Ce}, \cite{DLMS}, \cite{It}, \cite{NS}, \cite{IKY} and
\cite{BMS}.

All rings in this paper are commutative (but of course they are not assumed
to be noetherian). In the Introduction we are going to present the main
theorems of the paper, with only minimal explanations. Sections \ref{sec:back}
and \ref{sec:back-der} contain a much more detailed review of the necessary
background material.

Let us begin with a few words on adic completion.
Let $A$ be a ring, and let $\a \sub A$ be a finitely generated ideal.
For each $k \in \N$ define the ring $A_k := A / \a^{k + 1}$.
Given an $A$-module $M$, we write
$M_k := A_k \ot_A M$. With this notation, the {\em $\a$-adic completion} of $M$
is $\La_{\a}(M) := \lim_{\lto k} \msp{1} M_k$.
This is an $A$-linear functor from the category  $\cat{M}(A)$
of $A$-modules to itself (usually not exact on either side).
The setting above will be retained throughout the Introduction.

Let us now say a few words on {\em derived completion}.
The DG category of (unbounded)  complexes of $A$-modules
$M = \boplus_{i \in \Z} M^i$ is $\cat{C}(A)$.
The triangulated derived category of $A$-modules is $\cat{D}(A)$.
The $\a$-adic completion functor on modules extends to a DG functor
$\La_{\a} : \cat{C}(A) \to \cat{C}(A)$,
$\La_{\a}(M) := \boplus_{i \in \Z} \La_{\a}(M^i)$,
and there is a left derived functor
$\mrm{L} \La_{\a} : \cat{D}(A) \to \cat{D}(A)$.
A complex $M \in \cat{D}(A)$ is called {\em derived $\a$-adically complete} if
the canonical morphism $M \to \mrm{L} \La_{\a}(M)$ in $\cat{D}(A)$ is an
isomorphism.

The definition above, sometimes called {\em derived completeness in the
idealistic sense}, is not the same as  {\em derived completeness in the
sequential sense}. The latter concept is the one appearing in many modern
texts, e.g.\ \cite{SP} and \cite{BS}.
Generally speaking, the idealistic derived completion is more directly related
to plain completion; whereas sequential derived completion is easier to work
with. This is explained in Section \ref{sec:back-der}.
By results in \cite{PSY1} and \cite{Po2}, the two notions of derived
completion agree if and only if the ideal $\a$ is WPR (see Theorem
\ref{thm:270}). Since we shall mostly consider derived completion at WPR ideals
in our paper, the distinction between the idealistic and the sequential derived
completion will usually not matter.

An $A$-module $P$ is called {\em $\a$-adically flat} (see
\cite[Definition 4.2]{Ye2}), or
{\em $\a$-completely flat} (see \cite[Section 1.2]{BS}), if for every
$\a$-torsion $A$-module $N$ and every $i > 0$ we have $\opn{Tor}^A_i(P, N) = 0$.
It is known that the following three conditions are equivalent:
\begin{enumerate}
\item[$\rhd$] The $A$-module $P$ is $\a$-adically flat.

\item[$\rhd$]  For every $k \geq 0$ and every $i > 0$ the module
$\opn{Tor}^{\, A}_i(A_k, P)$ vanishes, and $P_k$ is a flat $A_k$-module.

\item[$\rhd$]  For every $i > 0$ the module
$\opn{Tor}^{\, A}_i(A_0, P)$ vanishes, and $P_0$ is a flat $A_0$-module.
\end{enumerate}
This holds without any finiteness assumptions on $A$, $\a$ or $P$.
See \cite[Theorem 4.3]{Ye4}, \cite[Section 4.1]{BMS} or
\cite[Section 1.2]{BS}.

When the ideal $\a$ is WPR and the module $P$ is complete, we can say more:

\begin{thm} \label{thm:501}
Let $A$ be a ring, let $\a$ be a weakly proregular ideal in $A$,
and let $P$ be an $\a$-adically complete $A$-module. For $k \geq 0$ define the
ring $A_k := A / \a^{k + 1}$ and the $A_k$-module $P_k := A_k \ot_A P$.
Then the following three conditions are equivalent:
\begin{enumerate}
\rmitem{i} The $A$-module $P$ is $\a$-adically flat.

\rmitem{ii} The functor $P \ot_{A} (-)$ is exact on $\a$-torsion $A$-modules.

\rmitem{iii} For every $k \geq 0$ the $A_k$-module $P_k$ is flat.
\end{enumerate}
\end{thm}

Theorem \ref{thm:501} is repeated as Theorem \ref{thm:502} in Section
\ref{sec:ad-flat}, and proved there.

For a graded $A$-module $M = \bigoplus_{i \in \Z} M^i$, its supremum is
$\opn{sup}(M) := \opn{sup} \msp{2} \{ i \in \Z \mid M^i \neq 0 \}$.
For a complex of $A$-modules $M$ we have two values to consider:
$\opn{sup}(M)$, which neglects the differential, and
$\opn{sup}(\opn{H}(M))$.
The complex $M$ is called bounded above if $\opn{sup}(M) < \infty$,
and it is called cohomologically bounded above
if $\opn{sup}(\opn{H}(M)) < \infty$.

\begin{thm} \label{thm:436}
Let $A$ be a ring, let $\a$ be a finitely generated ideal in $A$,
and let $P$ be a complex of $\a$-adically flat $A$-modules.
Assume either of these two conditions hold: either the complex $P$ is bounded
above, or the ideal $\a$ is weakly proregular.
Then the canonical morphism
$\eta^{\mrm{L}}_{P} : \mrm{L} \La_{\a}(P) \to \La_{\a}(P)$
in $\cat{D}(A)$ is an isomorphism.
\end{thm}

Theorem \ref{thm:436} is repeated as Theorem
\ref{thm:115} in Section \ref{sec:ad-flat}, and proved there.

Let us denote by $\cat{D}(A)_{\tup{$\a$-com}}$ the full subcategory of
$\cat{D}(A)$ on the derived $\a$-adically complete complexes.
It is a triangulated category. Theorems \ref{thm:232} and \ref{thm:233}
below describe the structure of the category $\cat{D}(A)_{\tup{$\a$-com}}$.

A complex of $A$-modules $P$ is called {\em $\a$-adically semi-free} if
$P = \La_{\a}(P')$, the $\a$-adic completion a semi-free
complex of $A$-modules $P'$.

\begin{thm} \label{thm:232}
Let $A$ be a ring, let $\a$ be a weakly proregular ideal in $A$, and
let $M \in \cat{D}(A)$. The following three conditions are equivalent:
\begin{itemize}
\rmitem{i} $M$ is a derived $\a$-adically complete complex.

\rmitem{ii} There is an isomorphism $M \cong P$
in $\cat{D}(A)$, where $P$ is an $\a$-adically semi-free complex of
$A$-modules.

\rmitem{iii} There is an isomorphism $M \cong N$ in $\cat{D}(A)$,
where $N$ is a complex of $\a$-adically complete $A$-modules.
\end{itemize}
Moreover, when these equivalent conditions hold, the $\a$-adically semi-free
complex $P$ in condition \tup{(ii)} can be chosen such that
$\opn{sup}(P) = \opn{sup}(\opn{H}(M))$.
\end{thm}

This theorem is a significant improvement upon \cite[Theorem 1.15]{PSY2},
in which the ring $A$ was assumed to be noetherian, and the complex $M$ was
assumed to have bounded above cohomology.
Theorem \ref{thm:232} is a combination of Theorems \ref{thm:155} and
\ref{thm:130} in the body of the paper, and the proofs are in Section
\ref{sec:recognize}.
See the end of the Introduction for a discussion of related work.

A complex of $A$-modules $P$ is called {\em $\a$-adically K-projective} if it
is a complex of $\a$-adically complete modules, and if for every acyclic complex
of $\a$-adically complete modules $N$, the complex
$\opn{Hom}_{A}(P, N)$ is acyclic.

\begin{thm} \label{thm:437}
Let $A$ be a ring, and let $\a$ be a weakly proregular ideal in $A$.
The following three conditions are equivalent for a complex $P$ of $\a$-adically
complete $A$-modules.
\begin{itemize}
\rmitem{i} $P$ is an $\a$-adically K-projective complex.

\rmitem{ii} There is a homotopy equivalence
$Q \to P$ in $\cat{C}_{\mrm{str}}(A)$, where $Q$ is an
$\a$-adically semi-free complex.

\rmitem{iii} For every complex of $\a$-adically complete $A$-modules
$M$, the canonical morphism
\[ \eta^{\mrm{R}}_{P, M} : \opn{Hom}_{A}(P, M) \to \opn{RHom}_{A}(P, M) \]
in $\cat{D}(A)$ is an isomorphism.
\end{itemize}
Furthermore, when these equivalent conditions hold, the $\a$-adically
semi-free complex $Q$ in condition {\rm (ii)} can be chosen such that
$\opn{sup}(Q) = \opn{sup}(\opn{H}(P))$.
\end{thm}

This is Theorem \ref{thm:401} in the body of the paper. The proof is in
Section \ref{sec:struct-cat-dc}.

The homotopy category of complexes of $A$-modules is $\cat{K}(A)$.
The categorical localization functor is
$\opn{Q} : \cat{K}(A) \to \cat{D}(A)$.

Denote by $\cat{K}(A)_{\tup{$\a$-sfr}}$ and
$\cat{K}(A)_{\tup{$\a$-kpr}}$ the full subcategories of $\cat{K}(A)$ on
the $\a$-adically semi-free and $\a$-adically K-projective complexes,
respectively. It is not hard to see that
$\cat{K}(A)_{\tup{$\a$-kpr}}$ is a full triangulated subcategory
of $\cat{K}(A)$; and Theorem \ref{thm:437} implies that
$\cat{K}(A)_{\tup{$\a$-sfr}}$ is also triangulated.
(See Remark \ref{rem:155} regarding the subtle difficulty with standard
cones of homomorphisms between $\a$-adically semi-free complexes.)

\begin{thm} \label{thm:233}
Let $A$ be a ring, and let $\a$ be a weakly proregular ideal in $A$.
Then the localization functor
$\opn{Q} : \cat{K}(A) \to \cat{D}(A)$
restricts to an equivalence of triangulated categories
$\opn{Q} : \cat{K}(A)_{\tup{$\a$-sfr}} \to \cat{D}(A)_{\tup{$\a$-com}}$.
\end{thm}

This theorem is an improvement of \cite[Theorem 1.19]{PSY2}, where the ring $A$
was noetherian, and the the target category
was the subcategory $\cat{D}^{-}(A)_{\tup{$\a$-com}}$.
Theorem \ref{thm:233} is repeated as Theorem \ref{thm:135}  and
proved there.
See the end of the Introduction for a discussion of related work.

The next theorem is a cohomological variant of the Nakayama Lemma.

\begin{thm} \label{thm:207}
Let $A$ be a ring, let $\a$ be a weakly proregular ideal in $A$, and
define the rings $\what{A} := \La_{\a}(A)$ and $A_0 := A / \a$.
Let $M$ be a derived $\a$-adically complete complex of $A$-modules,
with $\opn{sup}(\opn{H}(M)) = i_0$ for some $i_0 \in \Z$.
Suppose there is a number $r \in \N$ such that
$\opn{H}^{i_0}(A_0 \ot^{\mrm{L}}_{A} M)$
is generated by $\leq r$ elements as an $A_0$-module.
Then $\opn{H}^{i_0}(M)$ is a generated by $\leq r$ elements as an
$\what{A}$-module.
\end{thm}

This is repeated as Theorem \ref{thm:104} in the body of the paper.
The proof relies on Theorem \ref{thm:232}, and on the complete Nakayama theorem
for modules (Theorem \ref{thm:275}).
A weaker version of Theorem \ref{thm:207}
is \cite[Theorem 2.2]{PSY2}, where the ring $A$ was assumed to be noetherian.

\begin{cor} \label{cor:400}
In the setting of Theorem \ref{thm:207},
let $M$ and $N$ be derived $\a$-adically complete
complexes of $A$-modules, with
$\opn{sup}(\opn{H}(M)), \opn{sup}(\opn{H}(N)) \leq i_0$ for some
$i_0 \in \Z$.
Let $\phi : M \to  N$ be a morphism
in $\cat{D}(A)$.
The following two conditions are equivalent:
\begin{itemize}
\rmitem{i} The homomorphism
$\opn{H}^{i_0}(\phi) : \opn{H}^{i_0}(M) \to \opn{H}^{i_0}(N)$ is surjective.

\rmitem{ii} The homomorphism
\[ \opn{H}^{i_0}(\opn{id}_{A_0} \ot^{\mrm{L}}_{A} \msp{2} \phi) :
\opn{H}^{i_0}(A_0 \ot^{\mrm{L}}_{A} M) \to
\opn{H}^{i_0}(A_0 \ot^{\mrm{L}}_{A} N) \]
is surjective.
\end{itemize}
\end{cor}

The corollary is repeated as Corollary \ref{cor:412} and proved there.

There is a crucial difference between $\a$-adic completion in the noetherian
case and in the WPR case.
It is a classical fact that when $A$ is a noetherian ring, its
$\a$-adic completion $\what{A}$ is flat over $A$, and it is noetherian.
The flatness of $\what{A}$ over $A$ can fail when the ideal $\a$ is WPR but the
ring $A$ is not noetherian; see \cite[Theorem 7.2]{Ye4} for a counterexample.
The relevance of flatness is this: it is easy to prove (see Lemma
\ref{lem:510}) that if
$A \to B$ is a flat ring homomorphism, and $\a$ is a WPR ideal in $A$, then the
ideal $\b := B \cd \a$ in $B$ is WPR.

Nonetheless, we have:

\begin{thm} \label{thm:230}
Let $A$ be a ring, let $\a$ be a WPR ideal in $A$,
let $\what{A}$ be the $\a$-adic completion of $A$, and let
$\what{\a} := \what{A} \cd \a$, the ideal in $\what{A}$ generated by $\a$. Then
the ideal $\what{\a}$ is WPR.
\end{thm}

This theorem is repeated as Theorem \ref{thm:240}. The proof, in Section
\ref{sec:WPR-and-comp}, relies on the MGM equivalence (see Theorem
\ref{thm:271}). See discussion below regarding prior work of L. Positselski.

Here is a theorem that relies on Theorem \ref{thm:230}.

\begin{thm} \label{thm:516}
Let $A \to B$ be a flat ring homomorphism, and let $M$ be a flat $B$-module.
Let $\a \sub A$ be a weakly proregular ideal, and define the ideal
$\b := B \cd \a \sub B$. Let $\wh{B}$ be the $\b$-adic completion of $B$, with
ideal $\wh{\b} := \wh{B} \cd \b \sub B$, and let
$\wh{M}$ be the $\b$-adic completion of $M$.
Then $\wh{M}$ is a $\wh{\b}$-adically flat $\wh{B}$-module.
\end{thm}

One consequence of Theorem \ref{thm:516} is this:

\begin{cor} \label{cor:525}
Let $A \to B \to C$ be flat ring homomorphisms, with $A$ noetherian.
Given an ideal $\a \sub A$, let $\b := B \cd \a \sub B$ and
$\c := C \cd \a \sub C$ be the induced ideals, and let
$\wh{B}$ and $\wh{C}$ be the corresponding completions of $B$ and $C$.
Define the ideal $\wh{\b} := \wh{B} \cd \b \sub \wh{B}$.
Then $\wh{C}$ is $\wh{\b}$-adically flat over $\wh{B}$.
\end{cor}

The situation in the corollary arises naturally in certain aspects of perfectoid
theory. For instance when $A = \K[[t_1, \ldots, t_n]]$, the ring of powers
series over field of characteristic $p$, $C$ is the integral closure of $A$
(in an algebraic closure of the fraction field), and $B$ is the perfect closure
of $A$ in $C$.
Note that $\wh{B}$ and $\wh{C}$ are flat over $\wh{A}$, the $\a$-adic
completion of $A$, because the ring $\wh{A}$ is noetherian; see
\cite[Theorem 1.5]{Ye4}.

The theorem and the corollary are repeated as Theorem \ref{thm:510} and
Corollary \ref{cor:510}, where they are proved.

We do not know whether the theorems stated above remain true without the weak
proregularity condition.

To finish the Introduction, here is a discussion of related work. After showing
L. Positselski an early
version of our paper, containing Theorem \ref{thm:230}, he told us that
this result was already known to him. Indeed, it is stated in
\cite[Example 5.2(2)]{Po1}, with an indication how to prove it. As far as we
can tell, our proof (based on the MGM Equivalence and Lemma \ref{lem:156}) is
totally different.

After posting an early version of our paper online, we received a message from
J. Williamson, claiming that some of our results (Theorem \ref{thm:233}, and
the equivalence of conditions (i) and (iii) in Theorem \ref{thm:232})
can be deduced from results in
his joint paper \cite{PW}. While this claim might be true, we were not able to
verify it. The reason is that the theorems (and the proofs) in the paper
\cite{PW} are all in terms of Quillen equivalences between model categories. An
attempt to interpret these results in terms of derived categories and
triangulated functors, and then to compare them to our theorems, is quite
difficult.
We believe it is more appropriate to say that our Theorems \ref{thm:232} and
\ref{thm:233} are similar to some result in \cite{PW}.

It is worth mentioning that Theorem \ref{thm:230} was not known to the authors
of \cite{PW}, yet it was crucial to some of their main results. This forced
them to make statements contingent on Theorem \ref{thm:230} being true.


\section{Background Material on Completion of Modules}
\label{sec:back}

In this section we review some relevant facts and definitions about completion
and torsion, mostly taken from the paper \cite{Ye1}.
Recall that all rings in the paper are commutative.

Let us fix a ring $A$ and an ideal $\a \sub A$. Define the rings
$A_i := A / \a^{i + 1}$ for $i \in \N$.
The category of $A$-modules is denoted by $\cat{M}(A)$.
For an $A$-module $M$, we identify the $A_i$-modules
$M / (\a^{i + 1} \cd M) = A_i \ot_A M$.
The $\a$-adic completion of $M$ is
\begin{equation} \label{eqn:255}
\La_{\a}(M) := \lim_{\lto i} \msp{3} A_i \ot_A M .
\end{equation}
This is an $A$-linear functor from $\cat{M}(A)$ to itself, and there is
a functorial homomorphism $\tau_M : M \to \La_{\a}(M)$.
The module $M$ is called {\em $\a$-adically complete} if $\tau_M$ is an
isomorphism. If $\tau_M$ is injective, i.e.\ if
$\bigcap_i \a^i \cd M = 0$, then $M$ is called {\em $\a$-adically separated}.

The functor $\La_{\a}$ is neither left nor right exact.
If the ideal $\a$ is finitely generated, then the functor $\La_{\a}$
is idempotent, in the sense that for every module $M$ the $\a$-adic completion
$\La_{\a}(M)$ is $\a$-adically complete.
(See \cite[Example 1.8]{Ye1} for a counterexample when $\a$ is not finitely
generated.)
The full subcategory of $\cat{M}(A)$ on the $\a$-adically complete modules
is not an abelian subcategory.
(A substitute abelian category will be mentioned below.)

Let $M$ be an $A$-module. An element $m \in M$ is called an $\a$-torsion
element if $\a^i \cd m = 0$ for some $i \geq 1$. The set of
$\a$-torsion elements of $M$ is a submodule, called the $\a$-torsion submodule
of $M$, with notation $\Ga_{\a}(M)$.
By identifying
\[ \opn{Hom}_{A}({A_i}, M) = \{ m \in M \mid \a^{i + 1} \cd m = 0 \} , \]
we see that
\begin{equation} \label{eqn:256}
\Ga_{\a}(M) = \lim_{i \to} \msp{3} \opn{Hom}_{A}({A_i}, M) .
\end{equation}
This is an $A$-linear functor from $\cat{M}(A)$ to itself, and there is
a functorial homomorphism $\si_M : \Ga_{\a}(M) \to M$.
The module $M$ is called {\em $\a$-torsion} if $\Ga_{\a}(M) = M$, i.e.\ if
$\si_M$ is an isomorphism.

The functor $\Ga_{\a}$ is left exact. The full subcategory of
$\cat{M}(A)$ on the $\a$-torsion modules is denoted by
$\cat{M}_{\text{$\a$-tor}}(A)$.
It is an abelian subcategory. If $\a$ is finitely generated, then
$\cat{M}_{\text{$\a$-tor}}(A)$ is also closed under extensions, so it is a
thick abelian subcategory of $\cat{M}(A)$,

For the reasons mentioned above, from here on we shall assume this convention:

\begin{conv} \label{conv:280}
$A$ is a commutative ring, and $\a$ is a finitely generated ideal in
$A$. For $i \in \N$ we define the ring $A_i := A / \a^{i + 1}$. The $\a$-adic
completion of $A$ is the ring
$\what{A} := \lim_{\lto i} \msp{3} A_i$,
and $\what{\a} := \what{A} \cd \a$, the ideal in $\what{A}$ generated by $\a$.
\end{conv}

There are canonical ring homomorphisms $A \to A_i$ and
$\tau_A  : A \to \what{A}$. The ideal $\what{\a}$ is
finitely generated, the ring $\what{A}$ is $\what{\a}$-adically complete,
and the ring homomorphisms
$A_i \to \what{A} \msp{2} / \msp{2} \what{\a}^{\msp{2} i + 1}$ are
bijective.
If $M$ is an $\a$-adically complete $A$-module, then is has a unique
$\what{A}$-module structure extending the $A$-module structure.
Thus we may view completion as a functor
$\La_{\a} : \cat{M}(A) \to \cat{M}(\what{A})$.
We can also identify $\a$-adically complete $A$-modules with
$\what{\a}$-adically complete $\what{A}$-modules.
Similarly for $\a$-torsion $A$-modules.

Since completion of infinitely generated modules has a bit of subtlety (even
under Convention \ref{conv:280}), we provide the next elementary but useful
proposition.

\begin{prop} \label{prop:440}
Let $M$ and $N$ be $A$-modules.
\begin{enumerate}
\item Let $\what{M} := \La_{\a}(M)$. Then for every $i \in \N$
the homomorphism $\tau_{M, i} : A_i \ot_A M \to A_i \ot_A \what{M}$, that's
induced by $\tau_M$, is bijective.

\item Let $\what{M} := \La_{\a}(M)$. If $N$ is $\a$-torsion, then the
homomorphism $N \ot_A M \to N \ot_A \what{M}$ that's induced by $\tau_M$ is
bijective.

\item If $N$ is $\a$-adically complete, then the homomorphism
$\opn{Hom}_{A}(\La_{\a}(M), N) \to \lb \opn{Hom}_{A}(M, N)$
induced by $\tau_M$ is bijective.

\item If $M$ is $\a$-torsion, then the homomorphism
$\opn{Hom}_{A}(M, \Ga_{\a}(N)) \to \opn{Hom}_{A}(M, N)$
induced by $\si_N$ is bijective.

\item If $M$ and $N$ are both $\a$-adically complete (resp.\ $\a$-torsion),
then the homomorphism
$\opn{Hom}_{\what{A}}(M, N) \to \opn{Hom}_{A}(M, N)$,
corresponding to the ring homomorphism $A \to \what{A}$, is bijective.
\end{enumerate}
\end{prop}

\begin{proof}
(1) Let $M_i := A_i \ot_A M$, so
$\what{M} = \lim_{\lto i} \msp{3} M_i$.
By \cite[Theorem 2.8]{Ye4} the canonical homomorphism
$\pi_i : A_i \ot_A \what{M} \to M_i$ is bijective.
And $\pi_i \circ \tau_{M, i} = \opn{id}_{M_i}$.

\medskip \noindent
(2) Let $N^i := \opn{Hom}_{A}(A_i, N)$, so
$N =  \lim_{i \to}  N^i$.
Each $N^i$ is an $A_i$-module, so by item (1)
the homomorphism $N^i \ot_A M \to N^i \ot_A \what{M}$ is bijective.
Since $\ot_A$ respects direct limits, the assertion holds.

\medskip \noindent (3)
Let $\what{M}$ and $M_i$ be as above, and let $N_i := A_i \ot_A N$.
We are given that $N \cong \lim_{\lto i} \msp{3} N_i$.
By (1) we have $M_i \cong A_i \ot_A \what{M}$.
Then
\[ \opn{Hom}_{A}(\what{M}, N) \
\cong \lim_{\lto i} \msp{3} \opn{Hom}_{A_i}(M_i, N_i)
\cong \opn{Hom}_{A}(M, N) . \]

\medskip \noindent
(4) Since $M$ is torsion, every $\phi : M \to N$ factors through
$\Ga_{\a}(N)$.

\medskip \noindent
(5) Since $A_i \cong \what{A} / \msp{2} \what{\a}^{\msp{2} i + 1}$,
this is immediate from formula (\ref{eqn:255}) in the complete case, and
from formula (\ref{eqn:256}) in the torsion case.
\end{proof}

\begin{dfn} \label{dfn:145}
An $A$ module $P$ is called an {\em $\a$-adically free $A$-module} if it is
isomorphic to the $\a$-adic completion of a free $A$-module $P'$.
\end{dfn}

For some purposes it is useful to talk about {\em function modules}.
Given a set $Z$ and an $A$-module $M$, we denote by
$\opn{F}_{}(Z, M)$ the $A$-module of functions $\phi : Z \to M$.
Such a function $\phi$ can be viewed as a collection
$\bm = \{ m_z \}_{z \in Z}$ of elements of $M$, indexed by the set $Z$,
by letting $m_z := \phi(z)$.

In $\opn{F}_{}(Z, A)$ we have the submodule
$\opn{F}_{\mrm{fin}}(Z, A)$ of finitely supported functions
$\phi : Z \to A$. The $A$-module $\opn{F}_{\mrm{fin}}(Z, A)$
is free with basis the delta functions
$\{ \de_z \}_{z \in Z}$.

A function $\phi : Z \to \what{A}$ is called {\em $\a$-adically decaying} if
for every $i \geq 1$ the set
$\{ z \in Z \mid \phi(z) \notin \what{\a}^{\msp{2} i} \}$ is finite.
The $A$-module of $\a$-adically decaying functions is
denoted by $\opn{F}_{\mrm{dec}}(Z, \what{A})$.
It is known (see \cite[Corollaries 2.9 and 3.6]{Ye1}) that
$\opn{F}_{\mrm{dec}}(Z, \what{A})$ is the $\a$-adic completion of
$\opn{F}_{\mrm{fin}}(Z, A)$.
Every $\a$-adically free $A$-module $P$ is isomorphic to
$\opn{F}_{\mrm{dec}}(Z, \what{A})$ for a suitable set $Z$; indeed, if
$P \cong \La_{\a}(P')$ for some free module $P'$, and if
$P' \cong \opn{F}_{\mrm{fin}}(Z, A)$,
then we get an isomorphism $P \cong \opn{F}_{\mrm{dec}}(Z, \what{A})$.

An $\a$-adically complete $A$-module $M$ has
on it the $\a$-adic metric (see \cite[end of Section 1]{Ye1}), and $M$ is
complete for this metric (in the usual sense of metric spaces).

\begin{lem} \label{lem:256}
Let $M$ be an $\a$-adically complete $A$-module, and let $N \sub M$ be an
$A$-submodule that's closed for the $\a$-adic metric. Then there is a
surjective $A$-module homomorphism $\phi : P \to N$ from some  $\a$-adically
free $A$-module $P$.
\end{lem}

\begin{proof}
This is part of the proof of \cite[Lemma 1.20]{PSY2}, but the noetherian
assumption there is not needed. Choosing a generating collection
$\{ n_z \}_{z \in Z}$ for the $A$-module $N$,
we get a surjection
$\phi : P' \to N$,
where $P' := \opn{F}_{\mrm{fin}}(Z, A)$. The formula is
$\phi(\ba) = \sum_{z \in Z} a_z \cd n_z$
for a finitely supported collection
$\ba = \{ a_z \}_{z \in Z} \in P'$.
Let $P := \opn{F}_{\mrm{dec}}(Z, \what{A}) = \La_{\a}(P')$,
which is $\a$-adically free.
Since $M$ is complete, by Proposition \ref{prop:440}(2) we get a homomorphism
$\what{\phi} : P \to M$ extending $\phi : P' \to M$.
Because $N$ is closed in $M$, continuity implies that
$\phi(\ba) = \sum_{z \in Z} a_z \cd n_z \in N$ for all $\ba \in P$. Thus
$\what{\phi} : P \to N$ is a surjective homomorphism.
\end{proof}

\begin{prop} \label{prop:145}
Let  $M$ be an $\a$-adically complete $A$-module. Then there is an exact
sequence
\[ \cdots \to P^{-1} \xar{\d^{-1}_P} P^0 \xar{\pi} M \to 0 \]
of $A$-modules, such that all the $P^i$ are
$\a$-adically free $A$-modules.
\end{prop}

This is \cite[Lemma 1.20]{PSY2}, but the noetherian condition there is
superfluous. Here is a concise proof.

\begin{proof}
By Lemma \ref{lem:256} there is a surjection
$\pi : P^0 \to M$ from some $\a$-adically free $A$-module $P^0$.
The submodule $N_0 := \opn{Ker}(\pi) \sub P^0$ is closed, because $M$ is
$\a$-adically complete. Using Lemma \ref{lem:256}, there is a surjection
$\d^{-1}_P : P^{-1} \to N_0$ from some $\a$-adically free $A$-module $P^{-1}$.
And so on.
\end{proof}

While working on the present paper, we discovered an error in the proof
of \cite[Theorem 2.11]{Ye1}. Here is a correct proof of a slightly stronger
statement. Corollary \ref{cor:265} is a repetition of
\cite[Theorem 2.11]{Ye1}.

\begin{thm}[Complete Nakayama] \label{thm:275}
Let $A$ be a ring, let $\a$ be an ideal in $A$,  and let
$\phi : M \to N$ be an $A$-module homomorphism.
Assume that $M$ is $\a$-adically complete, and $N$ is $\a$-adically separated.
Define $A_0 := A / \a$ and $N_0 := A_0 \ot_A N$.
Let $\pi_0 : N \to N_0$ be the canonical surjection,
and let $\phi_0 := \pi_0 \circ \phi : M \to N_0$.
Then the following two conditions are equivalent:
\begin{itemize}
\rmitem{i}  $\phi : M \to N$ is surjective.

\rmitem{ii}  $\phi_0 : M \to N_0$ is surjective.
\end{itemize}
\end{thm}

\begin{proof}
We shall only treat the nontrivial implication.
The basic observation is this:
since $N_0 = N / (\a \cd N)$,
the surjectivity of $\phi_0$
means that $N = \phi(M) + \a \cd N$.
Therefore, given an arbitrary element
$n \in N$, we can find an element $m \in M$ such that
\begin{equation} \label{eqn:260}
n - \phi(m) \in \a \cd N .
\end{equation}

Now let's fix some element $n \in N$.
We are going to find a sequence
$m_0, m_1, \ldots$ of elements of $M$, such that
$m_k \in \a^{k} \cd M$, and for every $k \geq 0$ the formula
\begin{equation} \label{eqn:261}
n - \sum_{i = 0}^k \msp{5} \phi(m_i) \in \a^{k + 1} \cd N
\end{equation}
will hold. This will be done by induction on $k$.

For $k = 0$ we look at formula (\ref{eqn:260}), and define
$m_{0} := m$.

Now assume $k \geq 0$, and we have elements
$m_0, m_1, \ldots, m_k$ satisfying formula (\ref{eqn:261}).
Let $n' := n - \sum_{i = 0}^k \msp{2} \phi(m_i)$.
Since $n' \in \a^{k + 1} \cd N$, there are finitely many elements
$a_{i} \in \a^{k + 1}$ and $n'_{i} \in N$
such that
$n' = \sum_{i} \msp{2} a_i \cd n'_i$.
Using (\ref{eqn:260}) we can find elements $m'_i \in M$
such that $n'_i - \phi(m'_i) \in \a \cd N$.
Define
$m_{k + 1} := \sum_{i} \msp{2} a_i \cd m'_i \in \a^{k + 1} \cd M$.
Then
\[ n - \sum_{i = 0}^{k + 1} \msp{5} \phi(m_i) =
n' - \phi(m_{k + 1}) =
\sum_{i} \msp{2} a_i \cd (n'_i - \phi(m'_i)) \in \a^{k + 2} \cd N . \]
This finishes the inductive construction.

Because $M$ is $\a$-adically complete, the infinite sum
$m := \sum_{k = 0}^{\infty} m_k \in M$
converges. By formula (\ref{eqn:261}), the $\a$-adic continuity of $\phi$, and
the separatedness of $N$, we see that $\phi(m) = n$.
\end{proof}

Let $N$ be an $\a$-adically complete $A$-module. A collection
$\bn = \{ n_z \}_{z \in Z}$ of elements of $N$ is an
{\em $\a$-adic generating collection} if for every $n \in N$ there is an
$\a$-adically decaying collection
$\ba = \{ a_z \}_{z \in Z}$ of elements of $\what{A}$,
i.e.\ $\ba \in \opn{F}_{\mrm{dec}}(Z, \what{A})$, such that
$n = \sum_{z \in Z} a_z \cd n_z$.

\begin{cor} \label{cor:265}
In the setting of the theorem, assume that $N$ is also $\a$-adically complete.
the following conditions are equivalent for a
collection $\{ n_z \}_{z \in Z}$ of elements of $N$.
\begin{itemize}
\rmitem{i} The collection $\{ n_z \}_{z \in Z}$ is an $\a$-adic generating
collection of $N$.

\rmitem{ii} The collection $\{ \pi_0(n_z) \}_{z \in Z}$ is a
generating collection of $N_0$.
\end{itemize}
\end{cor}

\begin{proof}
Let $P := \opn{F}_{\mrm{dec}}(Z, \what{A})$, and let
$\phi : P \to N$ be the $A$-linear homomorphism
$\phi(\ba) := \sum_{z \in Z} a_z \cd n_z$.
Consider the theorem, with $M := P$.
Then condition (i) here is condition (i) in the theorem, and condition (ii)
here is condition (ii) in the theorem. So
(i) $\Leftrightarrow$ (ii).
\end{proof}

\section{Background Material on Derived Completion}
\label{sec:back-der}

First let us recall some categorical definitions and results, primarily
following \cite{PSY1}, \cite{Ye5} and \cite{Ye6}.
All rings are commutative.

Fix a ring $A$. Recall that $\cat{M}(A)$ is the category of
$A$-modules. It is
an $A$-linear abelian category. The DG category of complexes of $A$-modules is
$\cat{C}(A)$.
Given $M \in \cat{C}(A)$, we denote by $\opn{Z}^i(M)$,
$\opn{B}^i(M)$ and $\opn{H}^i(M)$ the modules of degree $i$ cocycles,
coboundaries, and cohomologies, respectively.
Taking direct sums we get the graded $A$-modules
$\opn{Z}(M) := \boplus_i \opn{Z}^i(M)$,
$\opn{B}(M) := \boplus_i \opn{B}^i(M)$ and
$\opn{H}(M) := \boplus_i \opn{H}^i(M)$.
These satisfy
$\opn{H}(M) =$ $\opn{Z}(M) \msp{2} / \msp{2} \opn{B}(M)$.

A homomorphism $\phi : M \to N$ in $\cat{C}(A)$ is called {\em strict}
if it has degree $0$ and it commutes with the differentials.
The strict subcategory of $\cat{C}(A)$, with all objects but only strict
homomorphisms, is $\cat{C}_{\mrm{str}}(A)$.  $\cat{C}_{\mrm{str}}(A)$ is an
$A$-linear abelian category, containing $\cat{M}(A)$ as a full abelian
subcategory, i.e.\ the complexes concentrated in degree $0$.
The homotopy category of $A$-modules is $\cat{K}(A)$, and it is an $A$-linear
triangulated category.
The projection functor (identity on objects, surjective on morphisms) is
$\opn{P} : \cat{C}_{\mrm{str}}(A) \to \cat{K}(A)$.
For complexes $M$ and $N$ we have
$\opn{Hom}_{\cat{C}(A)}(M, N) = \opn{Hom}_{A}(M, N)$,
$\opn{Hom}_{\cat{C}_{\mrm{str}}(A)}(M, N) = \opn{Z}^0(\opn{Hom}_{A}(M, N))$,
and
$\opn{Hom}_{\cat{K}(A)}(M, N) = \opn{H}^0(\opn{Hom}_{A}(M, N))$.
A homomorphism $\phi : M \to  N$ in $\cat{C}_{\mrm{str}}(A)$ is a homotopy
equivalence if and only if
$\opn{P}(\phi) : M \to N$ is an isomorphism in $\cat{K}(A)$.

The  derived category of $A$ is denoted by $\cat{D}(A)$, and it is
an $A$-linear triangulated category. $\cat{D}(A)$ is the categorical
localization of $\cat{K}(A)$ with respect to the quasi-isomorphisms, and the
localization functor is $\bar{\opn{Q}} : \cat{K}(A) \to \cat{D}(A)$,
an $A$-linear triangulated functor.
The composed functor
$\opn{Q} := \bar{\opn{Q}} \circ \opn{P} : \cat{C}_{\mrm{str}}(A)
\to \cat{D}(A)$
is an $A$-linear functor, and it is the identity on objects.
For more details on these matters see \cite[Chapters 3, 5, 7]{Ye5}.

Here is a review of the concepts of integer
interval, concentration, supremum, infimum and amplitude of a graded
$A$-module, all taken from \cite[Sections 12.1 and 12.4]{Ye5}.
For a graded $A$-module
$M = \bigoplus_{i \in \Z} M^i$,
its supremum is
$\opn{sup}(M) := \opn{sup} \{ \msp{2} i \mid M^i \neq 0 \} \in
\Z \cup \{ \pm \infty \}$.
The extreme cases are these:
$\opn{sup}(M) = \infty$ if and only if $M$ is unbounded above, and
$\opn{sup}(M) = -\infty$ if and only if $M = 0$.
The infimum $\opn{inf}(M)$ is defined analogously.
The amplitude of $M$ is
$\opn{amp}(M) := \opn{sup}(M) - \opn{inf}(M) \in \Z \cup \{ \pm \infty \}$.
If $M \neq 0$, and we let
$i_0 := \opn{inf}(M)$ and  $i_1 := \opn{sup}(M)$, then $M$ is concentrated in
the degree interval $[i_0, i_1]$.

For a complex of $A$-modules $M$ there are two distinct values:
$\opn{sup}(M)$, which refers to the supremum of the underlying graded
module, and $\opn{sup}(\opn{H}(M))$. Likewise for the other concepts.
We say that $M$ is bounded above (resp.\ cohomologically bounded above) if
$\opn{sup}(M) < \infty$ (resp.\ $\opn{sup}(\opn{H}(M)) < \infty$).

We shall use the stupid truncations from \cite[Definition 11.2.1]{Ye5} several
times. We shall also use the notions of cohomological displacement and
cohomological dimension of a functor from $\cat{D}(A)$ to itself, from
\cite[Section 12.4]{Ye5}.

A complex of $A$-modules $P$ is called {\em K-flat} if for every acyclic complex
$N$, the complex $P \ot_{A} N$ is acyclic. The complex $P$ is
called {\em K-projective} if for every acyclic complex $N$, the complex
$\opn{Hom}_{A}(P, N)$ is is acyclic. The complex $P$ is called {\em semi-free}
if it is a complex of free $A$-modules, which is either bounded above,
or else it admits a suitable filtration (see \cite[Definition 11.4.3]{Ye5}).
The logical implications are semi-free $\Rightarrow$ K-projective
$\Rightarrow$ K-flat.
Every complex $M$ admits a semi-free resolution, i.e.\ a quasi-isomorphism
$\rho : P \to M$ from a semi-free complex $P$, such that
$\opn{sup}(P) = \opn{sup}(\opn{H}(M))$; see \cite[Corollary 11.4.27]{Ye5}.

Let $\cat{K}(A)_{\tup{prj}}$ be the full subcategory of $\cat{K}(A)$ on the
K-projective complexes. Then the localization functor restricts to an
equivalence of triangulated categories
$\opn{Q} : \cat{K}(A)_{\tup{prj}} \to \cat{D}(A)$.
See \cite[Corollary 10.2.11]{Ye5}.

From here we assume that Convention \ref{conv:280} is in place,
so there is a finitely generated ideal $\a \sub A$.
The $A$-linear functor $\La_{\a} : \cat{M}(A) \to \cat{M}(A)$
extends to a DG functor
$\La_{\a} : \cat{C}(A) \to \cat{C}(A)$,
whose formula is
$\La_{\a}(M) := \boplus_{i \in \Z} \La_{\a}(M^i)$
for $M = \boplus_{i \in \Z} M^i \in \cat{C}(A)$.
There is an induced triangulated functor
$\La_{\a} : \cat{K}(A) \to \cat{K}(A)$,
and it admits a left derived functor
$\mrm{L} \La_{\a} : \cat{D}(A) \to \cat{D}(A)$,
with accompanying universal morphism
$\eta^{\mrm{L}} : \mrm{L} \La_{\a} \circ \bar{\opn{Q}} \to \bar{\opn{Q}} \circ
\, \La_{\a}$
of functors
$\cat{K}(A) \to \cat{D}(A)$.
The functor $\mrm{L} \La_{\a}$ is calculated using K-flat complexes.
There is a functorial morphism $\tau^{\mrm{L}}_M : M \to \mrm{L} \La_{\a}(M)$
in $\cat{D}(A)$ such that
$\eta^{\mrm{L}}_M \circ \tau^{\mrm{L}}_M = \opn{Q}(\tau_M)$ as morphisms
$M \to \La_{\a}(M)$.
These are shown in the following commutative diagram in $\cat{D}(A)$~:
\begin{equation} \label{eqn:535}
\begin{tikzcd} [column sep = 8ex, row sep = 6ex]
M
\ar[r, "{\tau^{\mrm{L}}_{M}}"]
\arrow[rr, bend left = 20, start  anchor = north,
end  anchor = north,  "{\opn{Q}(\tau^{}_{M})}"]
&
\mrm{L} \La_{\a}(M)
\ar[r, "{\eta^{\mrm{L}}_{M}}"]
&
\La_{\a}(M)
\end{tikzcd}
\end{equation}

To simplify notation, we shall often abuse notation a bit,
and write $\opn{Q}$ instead of $\bar{\opn{Q}}$ for the localization functor
$\cat{K}(A) \to \cat{D}(A)$.

The functor $\Ga_{\a} : \cat{M}(A) \to \cat{M}(A)$
admits a right derived functor
$\mrm{R} \Ga_{\a} : \cat{D}(A) \to \cat{D}(A)$,
with accompanying universal morphism of functors
$\eta^{\mrm{R}} : \Ga_{\a} \to \mrm{R} \Ga_{\a}$.
The functor $\mrm{R} \Ga_{\a}$ is calculated by K-injective complexes.
There is a functorial morphism
$\si^{\mrm{R}}_M : \mrm{R} \Ga_{\a}(M) \to M$
in $\cat{D}(A)$, such that
$\si^{\mrm{R}}_M \circ \eta^{\mrm{R}}_M = \opn{Q}(\si_M)$
as morphisms $\Ga_{\a}(M) \to M$. And there is a commutative diagram
like (\ref{eqn:535}.

\begin{dfn} \label{dfn:270}
Let $A$ be a ring and $\a$ an ideal in it.
\begin{enumerate}
\item A complex of $A$-modules $M$ is called {\em derived $\a$-adically
complete} if the morphism
$\tau^{\mrm{L}}_M : M \to \mrm{L} \La_{\a}(M)$
in $\cat{D}(A)$ is an isomorphism.
The full subcategory of $\cat{D}(A)$ on the derived $\a$-adically complete
complexes is denoted by $\cat{D}(A)_{\tup{$\a$-com}}$.

\item A complex of $A$-modules $M$ is called {\em derived $\a$-torsion}
if the morphism
$\si^{\mrm{R}}_M : \mrm{R} \Ga_{\a}(M) \to M$
in $\cat{D}(A)$ is an isomorphism.
The full subcategory of $\cat{D}(A)$ on the derived $\a$-torsion
complexes is denoted by $\cat{D}(A)_{\tup{$\a$-tor}}$.
\end{enumerate}
\end{dfn}

The categories $\cat{D}(A)_{\tup{$\a$-com}}$ and
$\cat{D}(A)_{\tup{$\a$-tor}}$ are triangulated.
In our previous papers (starting with \cite{PSY1}) we used the terms
``cohomologically $\a$-adically
complete'' and ``cohomologically $\a$-torsion''.
The reason for the updated names is that the properties described in Definition
\ref{dfn:270} are those of $M$ as an object of the derived category, and not
properties of the cohomology $\opn{H}(M)$.

In order to distinguish the definition above from other similar definitions in
the literature, we can be more precise and call these complexes
{\em derived $\a$-adically complete in the idealistic sense}, and
{\em derived $\a$-torsion in the idealistic sense}, respectively.
This distinction was studied in \cite{Ye6}, and below we provide a summary.

In order to define the other variants of derived completion and torsion, and
weak proregularity, we first have to introduce more concepts.

Recall that given an element $a \in A$, the associated {\em Koszul complex} is
\begin{equation} \label{eqn:333}
\opn{K}(A; a) :=
\bigl( \cdots \to 0 \to A \xar{a \cd (-)} A \to 0 \to \cdots \bigr)
\end{equation}
concentrated in degrees $-1$ and $0$.
Next, given a finite sequence $\ba = (a_1, \ldots, a_p)$ of elements in $A$,
with $p \geq 1$,
the associated Koszul complex of $\ba$ is
\begin{equation} \label{eqn:334}
\opn{K}(A; \ba) :=
\opn{K}(A; a_1) \ot_{A} \cdots \ot_{A} \opn{K}(A; a_p) .
\end{equation}
This is a complex of finite rank free $A$-modules, concentrated in degrees
$-p, \ldots, 0$. Therefore $\opn{K}(A; \ba)$ is a semi-free complex of
$A$-modules.
Moreover, $\opn{K}(A; \ba)$ is a {\em commutative semi-free DG $A$-ring} in the
sense of \cite[Definition 3.11(1)]{Ye3}.
As a graded ring we have
$\opn{K}(A; \ba)^{\natural} \cong A[x_1, \ldots, x_p]$, where the $x_i$ are
variables of degree $-1$.
If the sequence $\ba$ generates the ideal $\a$, then there is a canonical
$A$-ring isomorphism
$\opn{H}^0(\opn{K}(A; \ba)) \cong A_0 = A / \a$.
Therefore each $\opn{H}^i(\opn{K}(A; \ba))$ is an $A_0$-module.

For every $j_1 \geq j_0 \geq 1$ there is a homomorphism of complexes
$\opn{K}(A; a^{j_1}) \to \opn{K}(A; a^{j_0})$,
which is the identity in degree $0$, and multiplication by
$a^{j_1 - j_0}$ in degree $-1$.
For a sequence $\ba = (a_1, \ldots, a_p)$ of elements in $A$
we write $\ba^j := (a_1^j, \ldots, a_p^j)$.
Then for $j_1 \geq j_0 \geq 1$ there is a homomorphism of complexes
\begin{equation} \label{eqn:275}
\opn{K}(A; \ba^{j_1}) \to \opn{K}(A; \ba^{j_0}) .
\end{equation}
(In fact this is a DG $A$-ring homomorphism.)
In this way the collection
$\{ \opn{K}(A; \ba^{j}) \}_{j \geq 1}$
is an inverse system of complexes.

An inverse system of $A$-modules
$\{ N_j \}_{j \in \N}$ is called {\em pro-zero}, or is said to satisfy the
trivial ML condition, if for every $j_0$ there is some
$j_1 \geq j_0$ such that the homomorphism $N_{j_1} \to N_{j_0}$ is zero.

\begin{dfn} \label{dfn:255}
A sequence $\ba = (a_1, \ldots, a_p)$ of elements in a ring $A$ is called a
{\em weakly proregular sequence} if for every $i \leq -1$ the inverse system
$\bigl\{ \opn{H}^i \bigl( \opn{K}(A; \ba^{j}) \bigr) \bigr\}_{j \in \N}$
is pro-zero.
\end{dfn}

\begin{dfn} \label{dfn:256}
An ideal $\a$ in a ring $A$ is called a {\em weakly proregular ideal}
if it is generated by some weakly proregular sequence
$\ba = (a_1, \ldots, a_p)$.
\end{dfn}

The standard abbreviation for weakly proregular is WPR.
Grothendieck proved in \cite{LC} that when $A$ is a noetherian ring, every
ideal in it is WPR. By \cite[Corollary 6.3]{PSY1}, if $\a$ is a WPR ideal,
then every finite sequence $\ba = (a_1, \ldots, a_p)$ that generates $\a$ is a
WPR sequence.

Given a sequence $\ba = (a_1, \ldots, a_p)$ in $A$,
the {\em dual Koszul complex} is
\begin{equation} \label{eqn:330}
\opn{K}^{\vee}(A; \ba) := \opn{Hom}_A(\opn{K}(A; \ba), A) .
\end{equation}
The {\em infinite  dual Koszul complex}, also called the {\em augmented
\v{C}ech complex}, is
\begin{equation} \label{eqn:331}
\opn{K}^{\vee}_{\infty}(A; \ba) :=
\lim_{j \to} \opn{K}^{\vee}_{}(A; \ba^j) .
\end{equation}
The direct system is dual to the inverse system in formula
(\ref{eqn:275}).
The complex $\opn{K}^{\vee}_{\infty}(A; \ba)$ looks like this:
\begin{equation} \label{eqn:332}
\opn{K}^{\vee}_{\infty}(A; \ba) =
\bigl( \msp{4} 0 \to A \to \bigoplus_i A[a_i^{-1}] \to \cdots \to
A[(a_1 \cdots a_p)^{-1}] \to 0 \msp{4} \bigr) ,
\end{equation}
concentrated in degrees $0, \ldots, p$.
There is also the {\em telescope complex}
$\opn{Tel}(A; \ba)$, which is a complex of countable rank free $A$-modules,
concentrated in degrees $0, \ldots, p$; see \cite[Section 5]{PSY1}.
There is a canonical quasi-isomorphism
$\opn{Tel}(A; \ba) \to \opn{K}_{\infty}^{\vee}(A; \ba)$,
and a canonical homomorphism
$\opn{Tel}(A; \ba) \to A$ in $\cat{C}_{\mrm{str}}(A)$, called the augmentation.
For $M \in \cat{C}(A)$ let
\begin{equation} \label{eqn:337}
\tau^{\ba}_{M} : M \to \opn{Hom}_{A} \bigl( \opn{Tel}(A; \ba), M \bigr)
\end{equation}
be the homomorphism in $\cat{C}_{\mrm{str}}(A)$ induced by the augmentation.
See \cite[Equation (5.19)]{PSY1}.

Let $M$ be a complex of $A$-modules.
The {\em derived $\a$-adic completion in the sequential sense} of $M$  is
\begin{equation} \label{eqn:338}
\mrm{L}_{\mrm{seq}} \La_{\a} (M) :=
\opn{RHom}_A(\opn{K}_{\infty}^{\vee}(A; \ba), M) \cong
\opn{Hom}_A(\opn{Tel}(A; \ba), M) .
\end{equation}
The complex $M$ is called {\em derived $\a$-adically complete in
the sequential sense} if the morphism
\begin{equation} \label{eqn:339}
\opn{Q}(\tau^{\ba}_{M}) : M \to \mrm{L}_{\mrm{seq}} \La_{\a} (M)
\end{equation}
in $\cat{D}(A)$ is an isomorphism.
Similarly, The {\em derived $\a$-torsion in the sequential sense} of $M$  is
\begin{equation} \label{eqn:340}
\mrm{R}_{\mrm{seq}} \Ga_{\a} (M) :=
\opn{K}^{\vee}_{\infty}(A; \ba) \ot_A M \cong
\opn{Tel}(A; \ba) \ot_A M ,
\end{equation}
and the complex $M$ is called {\em derived $\a$-torsion in
the sequential sense} if the corresponding morphism
\begin{equation} \label{eqn:341}
\opn{Q}(\si^{\ba}_{M}) :  \mrm{R}_{\mrm{seq}} \Ga_{\a} (M) \to M
\end{equation}
in $\cat{D}(A)$ is an isomorphism.
It can be shown (see the proof of \cite[Corollary 6.2]{PSY1}) that the functors
$\mrm{L}_{\mrm{seq}} \La_{\a}$ and $\mrm{R}_{\mrm{seq}} \Ga_{\a}$ do not depend
(up to canonical isomorphisms) on the generating sequence $\ba$.
These are the notions of derived completion and torsion used in most current
texts, such as \cite{SP} and \cite{BS}.

According to \cite[Corollaries 4.26 and 5.25]{PSY1} there are functorial
commutative diagrams
\begin{equation} \label{eqn:270}
\begin{tikzcd} [column sep = 6ex, row sep = 6ex]
M
\ar[dr, "{\tau^{\mrm{L}}_{M}}"]
\ar[d, "{\opn{Q}(\tau^{\msp{2} \ba}_{M})}"']
\\
\mrm{L}_{\mrm{seq}} \La_{\a} (M)
\ar[r, "{v^{\mrm{L}}_{M}}"']
&
\mrm{L} \La_{\a}(M)
\end{tikzcd}
\qquad
\begin{tikzcd} [column sep = 5ex, row sep = 6ex]
\mrm{R}_{} \Ga_{\a} (M)
\ar[r, "{v^{\mrm{R}}_{M}}"]
\ar[dr, "{\si^{\mrm{R}}_{M}}"']
&
\mrm{R}_{\mrm{seq}} \Ga_{\a} (M)
\ar[d, "{\opn{Q}(\si^{\msp{2} \ba}_{M})}"]
\\
&
M
\end{tikzcd}
\end{equation}
in $\cat{D}(A)$.

\begin{thm}[\cite{PSY1}, \cite{Po2}, {\cite[Theorem 3.11]{Ye6}}]
\label{thm:270}
Let $A$ be a ring and let $\a$ be a finitely generated ideal in it.
The following three conditions are equivalent:
\begin{itemize}
\rmitem{i} The ideal $\a$ is weakly proregular.

\rmitem{ii} For every $M \in \cat{D}(A)$ the morphism $v^{\mrm{L}}_{M}$ in the
first diagram in (\ref{eqn:270}) is an isomorphism.

\rmitem{iii} For every $M \in \cat{D}(A)$ the morphism $v^{\mrm{R}}_{M}$ in the
second diagram in (\ref{eqn:270}) is an isomorphism.
\end{itemize}
\end{thm}

In plain words: weak proregularity of $\a$ is the necessary and sufficient
condition for the two kinds of derived $\a$-adic completion to agree; and the
same for derived torsion. See \cite[Section 3]{Ye6} for a more detailed
discussion.

\begin{cor} \label{cor:440}
Let $\a$ be a weakly proregular ideal in the ring $A$. Then the functor
$\mrm{L} \La_{\a} : \cat{D}(A) \to \cat{D}(A)$ has finite cohomological
dimension. More precisely, suppose
$\ba = (a_1, \ldots, a_p)$ is a generating sequence of the ideal $\a$. Then the
cohomological displacement of the functor $\mrm{L} \La_{\a}$ is contained in
the integer interval $[-p, 0]$.
\end{cor}

\begin{proof}
Let $T := \opn{Tel}(A; \ba)$ be the telescope complex associated to $\ba$.
It is a complex of free $A$-modules concentrated in the integer interval
$[0, p]$. According to Theorem \ref{thm:270}, or to
\cite[Corollary 5.25]{PSY1}, there is an isomorphism of
functors $\mrm{L} \La_{\a} \cong \opn{Hom}_A(T, -)$.
\end{proof}

Here is the main theorem of the paper \cite{PSY1}.

\begin{thm}[MGM Equivalence, {\cite[Theorem 1.1]{PSY1}}] \label{thm:271}
Let $\a$ be a weakly proregular ideal in the ring $A$. Then:
\begin{enumerate}
\item For every $M \in \cat{D}(A)$ one has
$\mrm{R} \Gamma_{\a} (M) \in \cat{D}(A)_{\a \tup{-tor}}$
and
$\mrm{L} \Lambda_{\a} (M) \in \cat{D}(A)_{\a \tup{-com}}$.

\item The functor
\[ \mrm{R} \Gamma_{\a} :
\cat{D}(A)_{\a \tup{-com}} \to
\cat{D}(A)_{\a \tup{-tor}} \]
is an equivalence, with quasi-inverse $\mrm{L} \Lambda_{\a}$.
\end{enumerate}
\end{thm}

The full subcategory $\cat{M}_{\tup{$\a$-com}}(A)$
of $\cat{M}(A)$ on the $\a$-adically
complete complexes is not abelian (it is not closed under cokernels, see
\cite[Example 3.20]{Ye1}). Positselski found the correct modification: the
category $\cat{M}_{\tup{$\a$-dcom}}(A)$, the full subcategory of
$\cat{M}(A)$ on the modules $M$ that are derived $\a$-adically complete in the
sequential sense; he calls them {\em $\a$-contramodules}, see
\cite[Section 1]{Po2}. The category
$\cat{M}_{\tup{$\a$-dcom}}(A)$
is a thick abelian subcategory of $\cat{M}(A)$.

The next theorem describes the full subcategories
$\cat{D}(A)_{\tup{$\a$-com}}$ and $\cat{D}(A)_{\tup{$\a$-tor}}$
of $\cat{D}(A)$ in terms of their cohomologies.

\begin{thm}[\cite{PSY1}, \cite{Po2}] \label{thm:276}
Let $A$ be a ring, let $\a$ be a weakly proregular ideal in $A$, and let
$M \in \cat{D}(A)$. Then:
\begin{enumerate}
\item The complex $M$ is derived $\a$-torsion if and only if its cohomology
modules $\opn{H}^i(M)$ are all $\a$-torsion modules.

\item The complex $M$ is derived $\a$-adically complete if and only if its
cohomology modules $\opn{H}^i(M)$ are all $\a$-contramodules.
\end{enumerate}
\end{thm}

Item (1) of the theorem is \cite[Corollary 4.3.2]{PSY1}, and item (2) is
\cite[Lemma 2.3]{Po1}.

\section{Adically Flat Modules}
\label{sec:ad-flat}

The goal of this section is to prove Theorems \ref{thm:502} and
\ref{thm:115}. Throughout the section we retain Convention \ref{conv:280};
so $A$ is a ring and $\a$ is a finitely generated ideal in it.

We begin by recalling this definition from \cite{Ye4}.

\begin{dfn} \label{dfn:260}
An $A$-module $P$ is called {\em $\a$-adically flat} if
$\opn{Tor}_i^A(P, N) = 0$ for every $\a$-torsion $A$-module $N$ and every
$i > 0$.
\end{dfn}

In other texts, such as \cite{BS}, the name {\em $\a$-completely flat} is used.
Of course if $P$ is flat then it is $\a$-adically flat.
But there are counterexamples: $\a$-adically flat modules that are not flat;
see \cite[Section 7]{Ye4}.

It is easy to see that an $A$-module $P$ is $\a$-adically flat
if and only if for every $\a$-torsion $A$-module $N$, the canonical morphism
$\eta^{\mrm{L}}_{P, N} :  P \ot^{\mrm{L}}_A N \to P \ot_A N$
in $\cat{D}(A)$ is an isomorphism.

The following theorem provides a description of $\a$-adically flat and
complete modules in the WPR case.

\begin{thm} \label{thm:502}
Let $A$ be a ring, let $\a$ be a weakly proregular ideal in $A$,
and let $P$ be an $\a$-adically complete $A$-module. For $k \geq 0$ define the
ring $A_k := A / \a^{k + 1}$ and the $A_k$-module $P_k := A_k \ot_A P$.
Then the following three conditions are equivalent:
\begin{enumerate}
\rmitem{i} The $A$-module $P$ is $\a$-adically flat.

\rmitem{ii} The functor $P \ot_{A} (-)$ is exact on $\a$-torsion $A$-modules.

\rmitem{iii} For every $k \geq 0$ the $A_k$-module $P_k$ is flat.
\end{enumerate}
\end{thm}

\begin{proof}
(i) $\Rightarrow$ (ii): Let
$0 \to N' \to N \to N'' \to 0$ be an exact sequence of torsion modules.
Since $\opn{Tor}^A_1(P, N'') = 0$, the sequence
\[ 0 \to P \ot_A N' \to P \ot_A N \to P \ot_A N'' \to 0 \]
is exact.

\medskip \noindent
(ii) $\Rightarrow$ (iii): For an $A_k$-module $N$ there is a canonical
isomorphism
$P \ot_A N \cong P_k \ot_{A_k} N$.
Hence $P_k \ot_{A_k} (-)$ is exact on $A_k$-modules.

\medskip \noindent
(iii) $\Rightarrow$ (i): This is the only implication that requires WPR.
The inverse system $\{ P_k \}_{k \in \N}$ is a flat $\a$-adic system, and
$P \cong \lim_{\lto k} \msp{1} P_k$. According to \cite[Theorem 6.9]{Ye4} the
$A$-module $P$ is $\a$-adically flat.
\end{proof}

\begin{thm} \label{thm:115}
Let $A$ be a ring, let $\a$ be a finitely generated ideal in $A$,
and let $P$ be a complex of $\a$-adically flat $A$-modules.
Assume that $P$ is bounded above, or that the ideal $\a$ is weakly proregular.
Then the canonical morphism
$\eta^{\mrm{L}}_{P} : \mrm{L} \La_{\a}(P) \to \La_{\a}(P)$
in $\cat{D}(A)$ is an isomorphism.
\end{thm}

The proof of the theorem comes after two lemmas.

\begin{lem} \label{lem:415}
Let $0 \to P' \to P \to P'' \to 0$
be an exact sequence of $A$-modules. If $P$ and $P''$ are $\a$-adically flat,
then so is $P'$.
\end{lem}

\begin{proof}
The proof is the same as for ordinary flatness.
\end{proof}

\begin{lem} \label{lem:416}
Let $P$ be an acyclic bounded above complex of $\a$-adically flat $A$-modules,
and let $N$ be an $\a$-torsion $A$-module.
Then the complex $P \ot_A N$ is acyclic.
\end{lem}

\begin{proof}
The proof is by the standard splicing argument.
Say $\opn{sup}(P) = i_0$.
Let $Q^i := \opn{Z}^i(P) = \opn{B}^i(P)$. The acyclic complex $P$
can be dissected into short exact sequences
\begin{equation} \label{eqn:415}
0 \to Q^i \xar{\msp{3} \mrm{inc} \msp{3}} P^i
\xar{\msp{3} \d_P \msp{3}} Q^{i + 1} \to  0 .
\end{equation}
By Lemma \ref{lem:415} and downward induction on $i \leq i_0$, the modules
$Q^i$ are all $\a$-adically flat. Hence the sequences
\begin{equation} \label{eqn:416}
0 \to Q^i \ot_A N \xar{} P^i  \ot_A N \xar{} Q^{i + 1}  \ot_A N \to  0
\end{equation}
are exact.
The exact sequences (\ref{eqn:416}) can now be spliced to yield the acyclic
complex $P \ot_A N$.
\end{proof}

\begin{proof}[Proof of Theorem {\rm \ref{thm:115}}]
The proof is similar to those of \cite[Corollary I.5.3]{RD} and
\cite[Lemma 16.1.5]{Ye5}.
However, since the functor $\La_{\a}$ is not right exact, special care is
required.

\medskip \noindent
Step 1. Here we assume that $P$ is a bounded above complex.
Choose a quasi-isomorphism $\phi : Q \to  P$, where $Q$ is a
bounded above complex of free $A$-modules.
Let $R$ be the standard cone of $\phi$. So $R$ is an acyclic bounded above
complex of $\a$-adically flat $A$-modules.

For every $k \geq 0$ consider the complex
$R_k := R \ot_{A} A_{k}$.
By Lemma \ref{lem:416}, with $N := A_k$, the complex $R_k$ acyclic.
We obtain an inverse system of acyclic complexes $\{ R_k \}_{k \geq 0}$
with surjective transition homomorphisms
$R_{k + 1} \to R_k$.
By the Mittag-Leffler argument (see
\cite[Proposition 1.12.4]{KS} or \cite[Corollary 11.1.8]{Ye5})
the complex
$\what{R} := \lim_{\lto k} R_k$ is acyclic.
But $\what{R}$ is isomorphic, in $\cat{C}_{\mrm{str}}(A)$,
to the standard cone of the homomorphism
$\La_{\a}(\phi) : \La_{\a}(Q) \to \La_{\a}(P)$,
and hence $\La_{\a}(\phi)$ is a quasi-isomorphism.

Now let's examine the following commutative diagram in $\cat{D}(A)$~:
\[ \begin{tikzcd} [column sep = 14ex, row sep = 6ex]
\mrm{L} \La_{\a}(Q)
\ar[r, "{\mrm{L} \La_{\a}(\opn{Q}(\phi))}", "{\simeq}"']
\ar[d, "{\eta^{\mrm{L}}_{Q}}"', "{\simeq}"]
&
\mrm{L} \La_{\a}(P)
\ar[d, "{\eta^{\mrm{L}}_{P}}"]
\\
\La_{\a}(Q)
\ar[r, "{\opn{Q}(\La_{\a}(\phi))}", "{\simeq}"']&
\La_{\a}(P)
\end{tikzcd} \]
(Notice that the upright $\opn{Q}$ is the localization functor, and the italic
$Q$ is the complex of modules.)
The two horizontal arrows are isomorphisms because
$\phi$ and $\La_{\a}(\phi)$ are quasi-isomorphisms.
The morphism $\eta^{\mrm{L}}_{Q}$ is an isomorphism because $Q$ is a
semi-free complex.
The conclusion is that
$\eta^{\mrm{L}}_{P}$ is an isomorphism.

\medskip \noindent
Step 2. Here $\a$ is WPR and $P$ can be unbounded.
By Corollary \ref{cor:440}, the functor $\mrm{L} \La_{\a}$ has cohomological
displacement $[-d, 0]$ for some $d \in \N$.

To prove that $\eta^{\mrm{L}}_{P} : \mrm{L} \La_{\a}(P) \to \La_{\a}(P)$
is an isomorphism, it suffices to prove that for every $i \in \Z$ the
homomorphism
$\opn{H}^i(\eta^{\mrm{L}}_{P}) : \opn{H}^i(\mrm{L} \La_{\a}(P)) \to
\opn{H}^i(\La_{\a}(P))$
is an isomorphism.

Fix an integer $i$. Let
$P' := \opn{stt}^{\leq i + d + 1}(P)$ and
$P'' := \opn{stt}^{\geq i + d + 2}(P)$, the stupid truncations.
There are distinguished triangles
$P'' \to P \to P' \xar{\triangle}$
and
$\mrm{L} \La_{\a}(P'') \to \mrm{L} \La_{\a}(P) \to
\mrm{L} \La_{\a}(P') \xar{\triangle}$
in $\cat{D}(A)$.
The cohomological concentrations of the complexes $P''$
and $\mrm{L} \La_{\a}(P'')$ are in the integer intervals
$[i + d + 2, \infty]$ and $[i + 2, \infty]$, respectively.
Therefore the homomorphisms
$\opn{H}^i(\La_{\a}(P)) \to \opn{H}^i(\La_{\a}(P'))$
and
$\opn{H}^i(\mrm{L} \La_{\a}(P)) \to \opn{H}^i(\mrm{L} \La_{\a}(P'))$
are isomorphisms.

Consider this commutative diagram
\[ \begin{tikzcd} [column sep = 6ex, row sep = 6ex]
\opn{H}^i(\mrm{L} \La_{\a}(P))
\ar[r, "{\simeq}"]
\ar[d, "{\opn{H}^i(\eta^{\mrm{L}}_{P})}"']
&
\opn{H}^i(\mrm{L} \La_{\a}(P'))
\ar[d, "{\opn{H}^i(\eta^{\mrm{L}}_{P'})}"]
\\
\opn{H}^i(\La_{\a}(P))
\ar[r, "{\simeq}"]
&
\opn{H}^i(\La_{\a}(P'))
\end{tikzcd} \]
in $\cat{M}(A)$. The horizontal arrows are isomorphisms by the previous
paragraph.
The complex $P'$ is a bounded above complex of $\a$-adically flat modules, so
by step 1 the homomorphism
$\opn{H}^i(\eta^{\mrm{L}}_{P'})$ is an isomorphism.
Hence
$\opn{H}^i(\eta^{\mrm{L}}_{P})$ is an isomorphism.
\end{proof}

\begin{rem}
We feel that the concept of adic flatness is not sufficiently understood.
Here are a few matters that we would like to settle.

First, to find a good definition of an {\em $\a$-adically K-flat
complex}, generalizing the usual definition of K-flat complexes.
Here are two reasonable definitions, for a complex $P$~:
\begin{enumerate}
\rmitem{i}For every complex of $\a$-torsion $A$-modules $N$, the canonical
morphism
$\eta^{\mrm{L}}_{P, N} :  P \ot^{\mrm{L}}_A N \to P \ot_A N$
in $\cat{D}(A)$ is an isomorphism.

\rmitem{ii}For every acyclic complex of $\a$-torsion modules $N$, the
complex $P \ot_A N$ is acyclic.
\end{enumerate}
It is quite easy to see that condition (i) implies condition (ii).
We do not know whether the reverse implication is true.

Next, we would like to have a description of $\a$-adically flat modules in
terms of $\a$-adically complete modules.

Lastly, we are wondering how to relate the $\a$-adic flatness of
$P$ to the complete tensor product operation
$P \msp{1} \wh{\ot}_{A} M := \La_{\a}(P \ot_A M)$.
\end{rem}

\section{Recognizing Derived Complete Complexes}
\label{sec:recognize}

The aim of this section is to prove Theorems
\ref{thm:155} and \ref{thm:130}, which are a repetition of
Theorem \ref{thm:232} from the Introduction.
These theorems provide criteria for a complex $M$ to be derived $\a$-adically
complete. In this section, in addition to Convention \ref{conv:280}, we also
assume that $\a$ is a weakly proregular ideal in $A$.

In \cite{Ye1} we defined $\a$-adically free $\a$-modules; this notion was
recalled in Definition \ref{dfn:145}.
We mentioned semi-free complexes in Section \ref{sec:back-der}.
Here is a combination of these notions.

\begin{dfn} \label{dfn:155}
A complex of $A$-modules $P$ is called {\em $\a$-adically semi-free} if
there is an isomorphism
$P \cong \La_{\a}(P')$ in $\cat{C}_{\mrm{str}}(A)$, where
$P'$ is a semi-free complex of $A$-modules.
\end{dfn}

This new concept is very good for calculations, as we shall now see; but there
is a categorical difficulty with it, as explained in Remark \ref{rem:155}.

\begin{thm} \label{thm:155}
Let $A$ be a ring, let $\a$ be a weakly proregular ideal in $A$, and
let $M \in \cat{D}(A)$. The following two conditions are equivalent:
\begin{itemize}
\rmitem{i} $M$ is a derived $\a$-adically complete complex.

\rmitem{ii} There is an isomorphism $M \cong P$
in $\cat{D}(A)$, where $P$ is an $\a$-adically semi-free complex of
$A$-modules.
\end{itemize}
Furthermore, when these equivalent conditions hold, the $\a$-adically semi-free
complex $P$ in condition \tup{(ii)} can be chosen such that
$\opn{sup}(P) = \opn{sup}(\opn{H}(M))$.
\end{thm}

\begin{proof}
(i) $\Rightarrow$ (ii):
Let $\psi : P' \to  M$ be a semi-free resolution in
$\cat{C}_{\mrm{str}}(A)$, such that
$\opn{sup}(P') = \opn{sup}(\opn{H}(M))$.
This resolution exists by \cite[Theorem 11.4.17 and Corollary 11.4.27]{Ye5}.
Define the complex $P := \La_{\a}(P')$,
which is an $\a$-adically semi-free complex of $A$-modules, and
$\opn{sup}(P) = \opn{sup}(P') = \opn{sup}(\opn{H}(M))$.

Consider this commutative diagram
\begin{equation} \label{eqn:405}
\begin{tikzcd} [column sep = 12ex, row sep = 6ex]
P'
\ar[d, "{\tau^{\mrm{L}}_{P'}}"]
\ar[r, "{\opn{Q}(\psi)}", "{\simeq}"']
\arrow[dd, bend right = 50, start anchor = west, end anchor = west,
"{\opn{Q}(\tau^{}_{P'})}"']
&
M
\ar[d, "{\tau^{\mrm{L}}_M}", "{\simeq}"']
\\
\mrm{L} \La_{\a}(P')
\ar[r, "{\mrm{L} \La_{\a}(\opn{Q}(\psi))}", "{\simeq}"']
\ar[d, "{\eta^{\mrm{L}}_{P'}}"', "{\simeq}"]
&
\mrm{L} \La_{\a}(M)
\\
\La_{\a}(P') = P
\end{tikzcd}
\end{equation}
in $\cat{D}(A)$.
The horizontal arrows are isomorphisms because $\psi$ is a quasi-isomorphism.
The morphism $\tau^{\mrm{L}}_M$ is an isomorphism because $M$ is derived
$\a$-adically  complete. And the morphism $\eta^{\mrm{L}}_{P'}$ is an
isomorphism because $P'$
is semi-free. We get an isomorphism
\[ (\tau^{\mrm{L}}_M)^{-1} \circ \mrm{L} \La_{\a}(\opn{Q}(\psi))
\circ (\eta^{\mrm{L}}_{P'})^{-1} : P \to M \]
in $\cat{D}(A)$.

\medskip \noindent
(ii) $\Rightarrow$ (i):
It is enough to prove that the complex $P$ is derived $\a$-adically complete.

By definition there is an isomorphism
$P \cong \La_{\a}(P')$ in $\cat{C}_{\mrm{str}}(A)$ for some semi-free complex
$P'$. Because $P'$ is semi-free there is an isomorphism
$\eta^{\mrm{L}}_{P'} : \mrm{L} \La_{\a}(P') \iso \La_{\a}(P')$
in $\cat{D}(A)$. Hence
$P \cong \mrm{L} \La_{\a}(P')$ in $\cat{D}(A)$.

Now we use that fact that $\a$ is weakly proregular.
According to \cite[Theorem 1.1]{PSY1},
the essential image of the functor
$\mrm{L} \La_{\a} : \cat{D}(A) \to \cat{D}(A)$
is $\cat{D}(A)_{\tup{$\a$-com}}$.
Since $P \cong \mrm{L} \La_{\a}(P')$, it follows that
$P \in \cat{D}(A)_{\tup{$\a$-com}}$.
\end{proof}

Here is a lemma needed for the proof of Theorem \ref{thm:130}.

\begin{lem} \label{lem:146}
Assume $\a$ is WPR.
If $M$ is an $\a$-adically complete $A$-module, then it is derived
$\a$-adically complete.
\end{lem}

\begin{proof}
We need to prove that the canonical morphism
$\tau^{\mrm{L}}_{M} :M \to \mrm{L} \La_{\a}(M)$
in $\cat{D}(A)$ is an isomorphism.

According to Proposition \ref{prop:145} there is a quasi-isomorphism
$\phi : P \to M$ in $\cat{C}_{\mrm{str}}(A)$, where $P$ is a nonpositive
complex of $\a$-adically free $A$-modules.
(As explained in Remark \ref{rem:155}, we do not know whether $P$ is an
$\a$-adically semi-free complex.)
By \cite[Theorem 5.3]{Ye4} the $A$-modules $P^i$ are all $\a$-adically
flat. Therefore, by Theorem \ref{thm:115},
$\eta^{\mrm{L}}_{P} : \mrm{L} \La_{\a}(P) \to \La_{\a}(P)$ is
an isomorphism.

Consider this commutative diagram
\begin{equation} \label{eqn:406}
\begin{tikzcd} [column sep = 12ex, row sep = 6ex]
P
\ar[d, "{\tau^{\mrm{L}}_P}"]
\ar[r, "{\opn{Q}(\phi)}", "{\simeq}"']
\arrow[dd, bend right = 50, start anchor = west, end anchor = west,
"{\opn{Q}(\tau^{}_P)}"', "{\simeq}"' near end]
&
M
\ar[d, "{\tau^{\mrm{L}}_M}"]
\\
\mrm{L} \La_{\a}(P)
\ar[r, "{\mrm{L} \La_{\a}(\opn{Q}(\phi))}", "{\simeq}"']
\ar[d, "{\eta^{\mrm{L}}_{P}}"', "{\simeq}"]
&
\mrm{L} \La_{\a}(M)
\\
\La_{\a}(P)
&
\end{tikzcd}
\end{equation}
in $\cat{D}(A)$. It is very similar to diagram (\ref{eqn:405}).
The two horizontal arrows are isomorphisms because
$\phi$ is a quasi-isomorphism.
The morphism $\opn{Q}(\tau^{}_P)$ is an isomorphism because
$\tau^{}_P : P \to \La_{\a}(P)$ is an isomorphism.
The conclusion is that both $\tau^{\mrm{L}}_{P}$ and
$\tau^{\mrm{L}}_{M}$ are isomorphisms.
\end{proof}

\begin{rem} \label{rem:150}
The converse of this lemma is false -- the $A$-module $M$ in
\cite[Example 3.20]{Ye1} is derived $\a$-adically complete, but it is not
$\a$-adically complete in the plain sense.

However, if an $A$-module $M$ is both derived $\a$-adically complete
and $\a$-adically separated, then $M$ is $\a$-adically complete
in the plain sense; see \cite{Ye2}.
\end{rem}

\begin{thm} \label{thm:130}
Let $A$ be a ring, let $\a$ be a weakly proregular ideal in $A$, and
let $M \in \cat{D}(A)$. The following two conditions are equivalent:
\begin{itemize}
\rmitem{i} $M$ is a derived $\a$-adically complete complex.

\rmitem{ii} There is an isomorphism $M \cong N$ in $\cat{D}(A)$,
where $N$ is a complex of $\a$-adically complete $A$-modules.
\end{itemize}
\end{thm}

\begin{proof}
(i) $\Rightarrow$ (ii): Use the implication
(i) $\Rightarrow$ (ii) in Theorem \ref{thm:155}, and take
$N := P$.

\medskip \noindent
(ii) $\Rightarrow$ (i):
We will prove that the complex $N$ is derived $\a$-adically
complete, namely that
$\tau^{\mrm{L}}_N : N \to \mrm{L} \La_{\a}(N)$ is an isomorphism.
It will be done in three steps, corresponding to the amplitude of $N$.

\medskip \noindent
Step 1. Here we assume the amplitude of $N$ is $\leq 0$, i.e.\
$N = N^i[-i]$ for some integer $i$.
By Lemma \ref{lem:146} the $A$-module $N^i$ is derived $\a$-adically
complete. Since $\cat{D}(A)_{\tup{$\a$-com}}$ is a full triangulated
subcategory of $\cat{D}(A)$, it follows that $N$ also belongs to it.

\medskip \noindent
Step 2. This is an inductive step.
Here we assume that $N$ has finite amplitude $l \geq 1$,
and that the assertion holds for every complex of complete modules
with amplitude $< l$.
The concentration of $N$ is an integer interval
$[i_0, i_1]$ with $i_0, i_1 \in \Z$ and $i_1 - i_0 = l$.
Let $N' := \opn{stt}^{\geq i_1}(N) = N^{i_1}[-i_1]$,
the stupid truncation above $i_1$, and let
$N'' := \opn{stt}^{\leq i_1 - 1}(N)$.
There is a distinguished triangle
$N' \to N \to N'' \xar{\msp{3} \triangle \msp{3}}$
in $\cat{D}(A)$.
Since the amplitudes of $N'$ and $N''$ are $< l$,
these belong to $\cat{D}(A)_{\tup{$\a$-com}}$.
Therefore also $N$ belongs to $\cat{D}(A)_{\tup{$\a$-com}}$.

\medskip \noindent
Step 3. Here $N$ is allowed to be unbounded in any direction.
We shall prove that for a fixed integer $i$, the homomorphism
$\opn{H}^i(\tau^{\mrm{L}}_N) : \opn{H}^i(N) \to \opn{H}^i(\mrm{L} \La_{\a}(N))$
is an isomorphism of $A$-modules.
The proof is very similar to step 2 in the proof of Theorem \ref{thm:115}.

Since $\a$ is WPR, the cohomological displacement of the functor
$\mrm{L} \La_{\a}$ is inside the integer interval $[-d, 0]$ for some
$d \in \N$.

Let
$N_1 := \opn{stt}^{\geq i - 1}(N)$, and let
$N_2 := \opn{stt}^{\leq i + 1 + d}(N_1)$.
There are homomorphisms $N_1 \to N$ and $N_1 \to N_2$ in
$\cat{C}_{\mrm{str}}(A)$, inducing isomorphism
$\opn{H}^i(N_1) \to \opn{H}^i(N)$ and
$\opn{H}^i(N_1) \to \opn{H}^i(N_2)$.

Next, let $N'_1 := \opn{stt}^{\leq i - 2}(N)$.
There is a distinguished triangle
$N_1 \to N \to N_1' \xar{\triangle}$ in $\cat{D}(A)$,
and an induced distinguished triangle
$\mrm{L} \La_{\a}(N_1) \to \mrm{L} \La_{\a}(N) \to
\mrm{L} \La_{\a}(N_1') \xar{\triangle}$.
The complex $N'_1$ has cohomological concentration in the integer interval
$[-\infty, i - 2]$. Due to the fact that the cohomological  displacement of the
functor $\mrm{L} \La_{\a}$ is inside the integer interval $[-d, 0]$,
it follows that $\mrm{L} \La_{\a}(N'_1)$ has cohomological concentration in the
integer interval $[-\infty, i - 2]$.
Therefore the homomorphism
$\opn{H}^i(\mrm{L} \La_{\a}(N_1)) \to \opn{H}^i(\mrm{L} \La_{\a}(N))$
is an isomorphism.

Define the complex $N'_2 := \opn{stt}^{\geq i + 2 + d}(N_1)$.
As above, there is a distinguished triangle
$\mrm{L} \La_{\a}(N'_2) \to \mrm{L} \La_{\a}(N_1) \to
\mrm{L} \La_{\a}(N_2) \xar{\triangle}$
in $\cat{D}(A)$. Due to the cohomological displacement, the complex
$\mrm{L} \La_{\a}(N'_2)$ has cohomological concentration in the
integer interval $[i + 2, \infty]$.
Therefore the homomorphism
$\opn{H}^i(\mrm{L} \La_{\a}(N_1)) \to \opn{H}^i(\mrm{L} \La_{\a}(N_2))$
is an isomorphism.

Consider these commutative diagrams in $\cat{M}(A)$. We just proved that the
horizontal arrows are isomorphisms.
\begin{equation} \label{eqn:442}
\begin{tikzcd} [column sep = 4ex, row sep = 6ex]
\opn{H}^i(N_1)
\ar[d, "{\opn{H}^i(\tau^{\mrm{L}}_{N_1}})"]
\ar[r, "{\simeq}"]
&
\opn{H}^i(N)
\ar[d, "{\opn{H}^i(\tau^{\mrm{L}}_{N})}"']
\\
\opn{H}^i(\mrm{L} \La_{\a}(N_1))
\ar[r, "{\simeq}"]
&
\opn{H}^i(\mrm{L} \La_{\a}(N))
\end{tikzcd}
\qquad
\begin{tikzcd} [column sep = 6ex, row sep = 6ex]
\opn{H}^i(N_1)
\ar[d, "{\opn{H}^i(\tau^{\mrm{L}}_{N_1}})"]
\ar[r, "{\simeq}"]
&
\opn{H}^i(N_2)
\ar[d, "{\opn{H}^i(\tau^{\mrm{L}}_{N_2})}"']
\\
\opn{H}^i(\mrm{L} \La_{\a}(N_1))
\ar[r, "{\simeq}"]
&
\opn{H}^i(\mrm{L} \La_{\a}(N_2))
\end{tikzcd}
\end{equation}
The complex $N_2$ is bounded. By step 2,
$\opn{H}^i(\tau^{\mrm{L}}_{N_2})$ is an isomorphism.
The second diagram shows that
$\opn{H}^i(\tau^{\mrm{L}}_{N_1})$ is an isomorphism, and then the first
diagram shows that
$\opn{H}^i(\tau^{\mrm{L}}_{N})$ is an isomorphism.
\end{proof}

\section{The Structure of the Category of Derived Complete Complexes}
\label{sec:struct-cat-dc}

In this section we study the category $\cat{D}(A)_{\tup{$\a$-com}}$
of derived $\a$-adically complete complexes. The main result is
Theorem \ref{thm:135}.
Theorem \ref{thm:401} is required for Theorem \ref{thm:135} to make sense, as
we explain in Remark \ref{rem:155}.
In addition to Convention \ref{conv:280}, we also assume
that $\a$ is a weakly proregular ideal in $A$.

K-projective complexes were recalled in Section \ref{sec:back-der}.
Here is an analogue for the adic setting -- a new definition.

\begin{dfn} \label{dfn:400}
A complex of $A$-modules $P$ is called {\em $\a$-adically
K-projective} if it satisfies these two conditions:
\begin{itemize}
\rmitem{a} $P$ is a complex of $\a$-adically complete $A$-modules.

\rmitem{b} Suppose $N$ is an acyclic complex of $\a$-adically
complete $A$-modules. Then the complex $\opn{Hom}_{A}(P, N)$ is acyclic.
\end{itemize}
\end{dfn}

\begin{prop} \label{prop:400}
If $P$ is an $\a$-adically semi-free complex, then it is an
$\a$-adically K-projective complex.
\end{prop}

\begin{proof}
Choose an isomorphism $P \cong  \La_{\a}(P')$ in $\cat{C}_{\mrm{str}}(A)$, for
some semi-free complex $P'$.
Let $N$ be an acyclic complex of $\a$-adically complete
$A$-modules. By Proposition \ref{prop:440}(3) there is an isomorphism
$\opn{Hom}_{A}(P, N) \cong \opn{Hom}_{A}(P', N)$
in $\cat{C}_{\mrm{str}}(A)$.
Since $P'$ is a K-projective complex of $A$-modules, the complex
$\opn{Hom}_{A}(P', N)$ is acyclic.
\end{proof}

\begin{lem} \label{lem:405}
Let $P$ and $Q$ be $\a$-adically K-projective complexes, and let
$\phi : P \to Q$ be a homomorphism in $\cat{C}_{\mrm{str}}(A)$.
The following conditions are equivalent:
\begin{itemize}
\rmitem{i} $\phi$ is a quasi-isomorphism.

\rmitem{ii} $\phi$ is a homotopy equivalence.
\end{itemize}
\end{lem}

\begin{proof}
The implication (ii) $\Rightarrow$ (i) it trivial.

The other implication is proved like the usual result
(cf.\ \cite[Corollary 10.2.14]{Ye3}), but taking care about completeness.
Suppose $\phi : P \to Q$ is a quasi-isomorphism in $\cat{C}_{\mrm{str}}(A)$.
We need to prove that the morphism
$\bar{\phi} := \opn{P}(\phi) : P \to Q$ in $\cat{K}(A)$ is an isomorphism.

Let $N$ be the standard cone of $\phi$, see \cite[Section 4.2]{Ye3}.
So $N$ is an acyclic complex of $\a$-adically complete $A$-modules.
Consider the standard short exact sequence
\begin{equation} \label{eqn:430}
 0 \to Q \to N \to P[1] \to 0
\end{equation}
in $\cat{C}_{\mrm{str}}(A)$,
see \cite[Section 4.2]{Ye5}.
This is a split exact sequence in the category $\cat{G}_{\mrm{str}}(A)$ of
graded $A$-modules.
Therefore the sequence
\begin{equation} \label{eqn:445}
0 \to \opn{Hom}_{A}(Q, Q) \to \opn{Hom}_{A}(Q, N) \to
\opn{Hom}_{A}(Q, P[1]) \to 0
\end{equation}
in $\cat{C}_{\mrm{str}}(A)$, obtained from (\ref{eqn:430}) by
applying the functor $\opn{Hom}_{A}(Q, -)$, is exact.
Since the complex $\opn{Hom}_{A}(Q, N)$ is acyclic, in the
long exact cohomology sequence associated to (\ref{eqn:445}),
the homomorphism
\[ \opn{H}^0 \bigl( \opn{Hom}_{A}(\opn{id}_{Q}, \phi) \bigr) :
\opn{H}^0 \bigl( \opn{Hom}_{A}(Q, P) \bigr) \to
\opn{H}^0 \bigl( \opn{Hom}_{A}(Q, Q) \bigr) \]
is an isomorphism. This means that the function
\[ \opn{Hom}_{\cat{K}(A)}(\opn{id}_{Q}, \bar{\phi}) :
\opn{Hom}_{\cat{K}(A)}(Q, P) \to \opn{Hom}_{\cat{K}(A)}(Q, Q) , \msp{5}
\bar{\psi} \mapsto \bar{\phi} \circ \bar{\psi} \]
is a bijection. Let
$\bar{\psi} : Q \to P$ be the morphism in $\cat{K}(A)$ satisfying
$\bar{\phi} \circ \bar{\psi} = \opn{id}_{Q}$.

It remains to prove that
$\bar{\psi} \circ \bar{\phi} = \opn{id}_{P}$.
By the same arguments as above, only now applying
$\opn{Hom}_{A}(P, -)$ to (\ref{eqn:430}), the function
\[ \opn{Hom}_{\cat{K}(A)}(\opn{id}_{P}, \bar{\phi}) :
\opn{Hom}_{\cat{K}(A)}(P, P) \to \opn{Hom}_{\cat{K}(A)}(P, Q)  , \msp{5}
\bar{\si} \mapsto \bar{\phi} \circ \bar{\si} \]
is a bijection.
Consider the morphisms $\bar{\psi} \circ \bar{\phi}$ and $\opn{id}_{P}$
in $\opn{Hom}_{\cat{K}(A)}(P, P)$. We have
\[ \bar{\phi} \circ (\bar{\psi} \circ \bar{\phi}) =
\opn{id}_{Q} \circ \msp{2} \bar{\phi} = \bar{\phi} =
\bar{\phi} \circ \opn{id}_{P} . \]
By canceling $\bar{\phi}$ we see that
$\bar{\psi} \circ \bar{\phi} = \opn{id}_{P}$.
\end{proof}

Here a useful characterization of $\a$-adically K-projective complexes.

\begin{thm} \label{thm:401}
Let $A$ be a ring, and let $\a$ be a weakly proregular ideal in $A$.
The following three conditions are equivalent for a complex $P$ of $\a$-adically
complete $A$-modules.
\begin{itemize}
\rmitem{i} $P$ is an $\a$-adically K-projective complex.

\rmitem{ii} There is a homotopy equivalence
$Q \to P$ in $\cat{C}_{\mrm{str}}(A)$, where $Q$ is an
$\a$-adically semi-free complex.

\rmitem{iii} For every complex of $\a$-adically complete $A$-modules $M$,
the canonical morphism
\[ \eta^{\mrm{R}}_{P, M} : \opn{Hom}_{A}(P, M) \to \opn{RHom}_{A}(P, M) \]
in $\cat{D}(A)$ is an isomorphism.
\end{itemize}
Furthermore, when these equivalent conditions hold, the $\a$-adically
semi-free complex $Q$ in condition {\rm (ii)} can be chosen such that
$\opn{sup}(Q) = \opn{sup}(\opn{H}(P))$.
\end{thm}

\begin{proof}
(i) $\Rightarrow$ (ii):
According to Theorem \ref{thm:130} the complex $P$ is derived $\a$-adically
complete. Choose some semi-free resolution
$\rho : Q' \to P$ in $\cat{C}_{\mrm{str}}(A)$,
such that $\opn{sup}(Q') = \opn{sup}(\opn{H}(P))$.
Define the $\a$-adically semi-free complex $Q := \La_{\a}(Q')$, and the
homomorphism
$\what{\rho} := \La_{\a}(\rho) : Q \to P$. We want to prove that
$\what{\rho} : Q \to P$ is a homotopy equivalence.

Let's examine the next commutative diagram
\begin{equation} \label{eqn:407}
\begin{tikzcd} [column sep = 12ex, row sep = 6ex]
Q'
\ar[d, "{\tau^{\mrm{L}}_{Q'}}"]
\ar[r, "{\opn{Q}(\rho)}", "{\simeq}"']
\arrow[dd, bend right = 50, start anchor = west, end anchor = west,
"{\opn{Q}(\tau^{}_{Q'})}"']
&
P
\ar[d, "{\tau^{\mrm{L}}_P}"', "{\simeq}"]
\arrow[dd, bend left = 50, start anchor = east, end anchor = east,
"{\opn{Q}(\tau^{}_{P})}", "{\simeq}"']
\\
\mrm{L} \La_{\a}(Q')
\ar[r, "{\mrm{L} \La_{\a}(\opn{Q}(\rho))}", "{\simeq}"']
\ar[d, "{\eta^{\mrm{L}}_{Q'}}"', "{\simeq}"]
&
\mrm{L} \La_{\a}(P)
\ar[d, "{\eta^{\mrm{L}}_{P}}"]
\\
\La_{\a}(Q') = Q
\ar[r, "{\opn{Q}(\what{\rho})}"]
&
\La_{\a}(P) = P
\end{tikzcd}
\end{equation}
in $\cat{D}(A)$.
(Note that upright $\opn{Q}$ is the functor, and italic $Q$ is the complex.)
The morphisms $\opn{Q}(\rho)$ and $\mrm{L} \La_{\a}(\opn{Q}(\rho))$
are isomorphisms because $\rho$ is a quasi-iso\-morphism.
The morphism $\tau^{\mrm{L}}_P$ is an isomorphism because $P$ is derived
$\a$-adically complete.
The morphism $\opn{Q}(\tau_P)$ is an isomorphism because $\tau_P$ is an
isomorphism.
The morphism $\eta^{\mrm{L}}_{Q'}$ is an isomorphism because $Q'$ is
semi-free. By the commutativity of the diagram it follows that
all the arrows in it are isomorphisms.
In particular, $\opn{Q}(\what{\rho}) : Q \to P$ is an isomorphism.
Thus $\what{\rho} : Q \to P$ is a quasi-isomorphism.
According to Proposition \ref{prop:400} and
Lemma \ref{lem:405}, $\what{\rho}$ is a homotopy
equivalence.

\medskip \noindent
(ii) $\Rightarrow$ (iii):
Let $\phi : Q \to P$ be a homotopy equivalence in $\cat{C}_{\mrm{str}}(A)$ from
an $\a$-adically semi-free complex $Q$. Let
$\psi : \La_{\a}(Q') \to Q$ be an isomorphism in $\cat{C}_{\mrm{str}}(A)$,
where $Q'$ is a semi-free complex.
We obtain this commutative diagram in $\cat{D}(A)$.
\begin{equation} \label{eqn:531}
\begin{tikzcd} [column sep = 3ex, row sep = 6ex]
\opn{Hom}_A(P, M)
\ar[d, "{\eta^{\mrm{R}}_{P, M}}"]
\ar[r, "{\al}"]
&
\opn{Hom}_A(Q, M)
\ar[d, "{\eta^{\mrm{R}}_{Q, M}}"]
\ar[r, "{\beta}"]
&
\opn{Hom}_A(\La_{\a}(Q'), M)
\ar[d, "{\eta^{\mrm{R}}_{\La_{\a}(Q'), M}}"]
\ar[r, "{\ga}"]
&
\opn{Hom}_A(Q', M)
\ar[d, "{\eta^{\mrm{R}}_{Q', M}}"]
\\
\opn{RHom}_A(P, M)
\ar[r, "{\al^{\mrm{R}}}"]
&
\opn{RHom}_A(Q, M)
\ar[r, "{\beta^{\mrm{R}}}"]
&
\opn{RHom}_A(\La_{\a}(Q'), M)
\ar[r, "{\ga^{\mrm{R}}}"]
&
\opn{RHom}_A(Q', M)
\end{tikzcd}
\end{equation}
Here
$\al := \opn{Q}(\opn{Hom}_A(\phi, \opn{id}_M))$,
$\al^{\mrm{R}} := \opn{RHom}_A(\opn{Q}(\phi), \opn{id}_M)$, etc.
The morphisms $\al$ and $\beta$ are isomorphisms because
$\phi$ is a homotopy equivalence and $\psi$ is an isomorphism.
For the same reason $\al^{\mrm{R}}$ and $\beta^{\mrm{R}}$ are isomorphisms.
The morphism $\ga$ is an isomorphism by Proposition \ref{prop:440}(3).

We shall now prove that the morphism
$\ga^{\mrm{R}} =  \opn{RHom}_{A}(\opn{Q}(\tau_{Q'}), \opn{id}_M)$
in $\cat{D}(A)$ is an isomorphism.
Consider this commutative diagram in $\cat{D}(A)$~:

\begin{equation} \label{eqn:530}
\begin{tikzcd} [column sep = 16ex, row sep = 6ex]
\opn{RHom}_A(\La_{\a}(Q'), M)
\ar[dr, "{\ga^{\mrm{R}}}"]
\ar[d, "{\opn{RHom}_A(\eta^{\mrm{L}}_{Q'}, \opn{id}_M)}"']
\\
\opn{RHom}_A(\mrm{L} \La_{\a}(Q'), M)
\ar[r, "{\opn{RHom}_A(\tau^{\mrm{L}}_{Q'}, \opn{id}_M)}"']
&
\opn{RHom}_A(Q', M)
\end{tikzcd}
\end{equation}
According to Theorem \ref{thm:130} the morphism
$\tau^{\mrm{L}}_{M} : M \to \mrm{L} \La_{\a}(M))$
is an isomorphism. This means that in diagram (\ref{eqn:530}) we can replace
$M$ with $\mrm{L} \La_{\a}(M)$.
Because $Q'$ is a semi-free complex, the morphism
$\eta^{\mrm{L}}_{Q'} : \mrm{L} \La_{\a}(Q') \to \La_{\a}(Q')$
is an isomorphism.
Therefore
$\opn{RHom}_A(\eta^{\mrm{L}}_{Q'}, \opn{id}_M)$
is an isomorphism too.
The last isomorphism in \cite[Erratum, Theorem 9]{PSY1} implies that
$\opn{RHom}_A(\tau^{\mrm{L}}_{Q'}, \opn{id})$
is an isomorphism. Conclusion: $\ga^{\mrm{R}}$ is an isomorphism.

Going back to diagram (\ref{eqn:531}), we see that all the horizontal
arrows in it are isomorphisms.
Again using the fact that $Q'$ is semi-free, it follows that
$\eta^{\mrm{R}}_{Q', M}$ is an isomorphism.
We conclude that all arrows in the
diagram are isomorphisms, and in particular this is true for
$\eta^{\mrm{R}}_{P, M}$.

\medskip \noindent
(iii) $\Rightarrow$ (i):
Take an acyclic complex $N$ of $\a$-adically complete modules.
Since $N$ is acyclic, we have $\opn{RHom}_{A}(P, N) = 0$ in $\cat{D}(A)$.
Since $N$ is a complex of $\a$-adically complete modules,
condition (iii) says that $\opn{Hom}_{A}(P, N) = 0$ in $\cat{D}(A)$.
This means that $\opn{Hom}_{A}(P, N)$ is an acyclic complex. Thus $P$ is
$\a$-adically K-projective.
\end{proof}

\begin{dfn} \label{dfn:160}
The full subcategory of $\cat{K}(A)$ on the
$\a$-adically semi-free (resp.\ $\a$-adically K-projective) complexes is denoted
by $\cat{K}(A)_{\tup{$\a$-sfr}}$ (resp.\ $\cat{K}(A)_{\tup{$\a$-kpr}}$).
\end{dfn}

\begin{prop} \label{prop:430}
The categories $\cat{K}(A)_{\tup{$\a$-kpr}}$
and $\cat{K}(A)_{\tup{$\a$-sfr}}$
are full triangulated subcategories of $\cat{K}(A)$,
and the inclusion
$\cat{K}(A)_{\tup{$\a$-sfr}} \to \cat{K}(A)_{\tup{$\a$-kpr}}$
is an equivalence of triangulated categories.
\end{prop}

\begin{proof}
It is clear that the full subcategory of $\cat{C}_{\mrm{str}}(A)$ on the
$\a$-adically K-projective complexes is closed under translations and finite
direct sums. Suppose $\phi : P \to Q$ is a homomorphism in
$\cat{C}_{\mrm{str}}(A)$ between two $\a$-adically K-projective complexes, and
let $R$ be the standard cone of $\phi$. So $R$ is a complex of $\a$-adically
complete modules. We have the standard short exact sequence
$0 \to Q \to R \to P[1] \to 0$
in $\cat{C}_{\mrm{str}}(A)$, which is split in
$\cat{G}_{\mrm{str}}(A)$. Given an acyclic complex of $\a$-adically complete
modules $N$, the sequence
\[ 0 \to \opn{Hom}_{A}(P[1], N) \to \opn{Hom}_{A}(R, N) \to
\opn{Hom}_{A}(Q, N) \to 0 \]
in $\cat{C}_{\mrm{str}}(A)$ is exact. Since the complexes
$\opn{Hom}_{A}(P[1], N)$ and $\opn{Hom}_{A}(Q, N)$ are acyclic, so is
$\opn{Hom}_{A}(R, N)$. We see that $R$ is also $\a$-adically K-projective.
Conclusion: the category $\cat{K}(A)_{\tup{$\a$-kpr}}$ is a full triangulated
subcategory of $\cat{K}(A)$.

According to Proposition \ref{prop:400} and Theorem \ref{thm:401},
$\cat{K}(A)_{\tup{$\a$-kpr}}$
is the closure of $\cat{K}(A)_{\tup{$\a$-sfr}}$ under isomorphisms
inside $\cat{K}(A)$. This implies that
$\cat{K}(A)_{\tup{$\a$-sfr}}$ is also triangulated, and the inclusion
$\cat{K}(A)_{\tup{$\a$-sfr}} \to \cat{K}(A)_{\tup{$\a$-kpr}}$
is an equivalence.
\end{proof}

See Remark \ref{rem:155} regarding the problem of cones with $\a$-adically
semi-free complexes, and the importance of this proposition.

The next theorem is a repetition of Theorem \ref{thm:233} from the
Introduction. It is a generalization of \cite[Theorem 1.19]{PSY2} in
two ways: first, we replace the noetherian condition on $A$ by the WPR
condition on $\a$; and second, we allow for cohomologically unbounded
complexes.

\begin{thm} \label{thm:135}
Let $A$ be a ring, and let $\a$ be a weakly proregular ideal in $A$.
Then the localization functor
$\opn{Q} : \cat{K}(A) \to \cat{D}(A)$
restricts to an equivalence of triangulated categories
\[ \tag{$*$}  \opn{Q} : \cat{K}(A)_{\tup{$\a$-sfr}} \to
\cat{D}(A)_{\tup{$\a$-com}} . \]
\end{thm}

\begin{proof}
By Theorem \ref{thm:155}, the functor $\opn{Q}$ in formula ($*$) is essentially
surjective on objects.
Let us prove that it is fully faithful.
Take any two complexes $P, Q \in \cat{K}(A)_{\tup{$\a$-sfr}}$.
Since $P$ is an $\a$-adically K-projective complex and $Q$ is a complex of
$\a$-adically complete modules, according to Theorem \ref{thm:401} the morphism
\[ \eta^{\mrm{R}}_{Q, P} : \opn{Hom}_{A}(P, Q) \to
\opn{RHom}_{A}(P, Q) \]
in $\cat{D}(A)$ is an isomorphism.
Taking $\opn{H}^0$, we obtain an isomorphism
\[ \opn{H}^0(\eta^{\mrm{R}}_{Q, P}) :
\opn{H}^0(\opn{Hom}_{A}(P, Q)) \to
\opn{H}^0(\opn{RHom}_{A}(P, Q)) . \]
Interpreting these as the morphism sets in
$\cat{K}(A)$ and $\cat{D}(A)$ respectively, the isomorphism
$\opn{H}^0(\eta^{\mrm{R}}_{Q, P})$ becomes $(*)$.
\end{proof}

\begin{rem} \label{rem:155}
As we have already seen in Section \ref{sec:recognize},
$\a$-adically semi-free complexes are very good for calculations.
But they pose a thorny descent problem that we now describe.

Suppose $P$ and $Q$ are $\a$-adically free modules, and $\phi : P \to Q$
is a homomorphism in $\cat{M}(A)$. We do not know whether $\phi$
descends to free complexes, namely whether there exists
a homomorphism $\phi' : P' \to Q'$ between free $A$-modules,
and a commutative diagram
\[ \begin{tikzcd} [column sep = 10ex, row sep = 4ex]
\La_{\a}(P')
\ar[d, "{\simeq}"']
\ar[r, "{\La_{\a}(\phi')}"]
&
\La_{\a}(Q')
\ar[d, "{\simeq}"]
\\
P
\ar[r, "{\phi}"]
&
Q
\end{tikzcd} \]
in $\cat{M}(A)$ with vertical isomorphisms.

For this reason, we do not know if every bounded above complex of
$\a$-adically free modules $P$ is an $\a$-adically semi-free complex.
We also do not know if the full subcategory of $\cat{C}_{\mrm{str}}(A)$ on
the $\a$-adically semi-free complexes is closed under taking standard cones.

For Theorem \ref{thm:135} to make sense, we need we need the category
$\cat{K}(A)_{\tup{$\a$-sfr}}$ to be triangulated.
We prove this in Proposition \ref{prop:430}, using $\a$-adically K-projective
complexes.
\end{rem}

\section{A Derived Complete Nakayama Theorem}
\label{sec:nakayama}

The next theorem is a repetition of Theorem \ref{thm:207} from the
Introduction. It is a generalization of \cite[Theorem 2.2]{PSY2} from the
noetherian case to the WPR case.

\begin{thm} \label{thm:104}
Let $A$ be a ring, let $\a$ be a weakly proregular ideal in $A$, let
$\what{A}$ be the $\a$-adic completion of $A$, and let
and $A_0 := A / \a$.
Let $M$ be a derived $\a$-adically complete complex of $A$-modules.
Suppose that there are numbers $i_0 \in \Z$ and $r \in \N$ such that
$\opn{sup}(\opn{H}(M)) = i_0$,
and $\opn{H}^{i_0}(A_0 \ot^{\mrm{L}}_{A} M)$
is generated by $\leq r$ elements as an $A_0$-module.
Then $\opn{H}^{i_0}(M)$ is generated by $\leq r$ elements as an
$\what{A}$-module.
\end{thm}

\begin{proof}
The proof is almost identical to the proof of \cite[Theorem 2.2]{PSY2},
but we need Theorem \ref{thm:155} to pass from the
noetherian setting to the WPR setting.

We can assume that $i_0 = 0$. By Theorem \ref{thm:155} we can replace $M$
with a nonpositive complex of $\a$-adically free $A$-modules $P$.
We are given that
$L_0 := \opn{H}^0(A_0 \ot^{\mrm{L}}_{A} P)$
is generated as an $A_0$-module by $\leq r$ elements, and we must prove that
$L := \opn{H}^0(P)$ is generated as an $\what{A}$-module by $\leq r$ elements.

By definition there are exact sequences of $\what{A}$-modules
\begin{equation} \label{eqn:447}
P^{-1} \xar{\d} P^0 \to L \to 0
\end{equation}
and
\begin{equation} \label{eqn:448}
A_0 \ot^{}_{A} P^{-1} \xar{\d_0} A_0 \ot^{}_{A} P^0 \to
\opn{H}^0(A_0 \ot^{}_{A} P) \to 0
\end{equation}
Here $\d$ and $\d_0$ are the differentials of the complexes
$P$ and $A_0 \ot^{}_{A} P$.
According to the K\"unneth tricks, see \cite[Theorems 1 and 7]{Ye7}, there are
canonical $A$-module isomorphisms
\begin{equation} \label{eqn:449}
\opn{H}^0(A_0 \ot^{}_{A} P) \cong A_0 \ot^{}_{A} \opn{H}^0(P)
\cong \opn{H}^0(A_0 \ot^{\mrm{L}}_{A} P) .
\end{equation}
Therefore we can transform (\ref{eqn:448}) into the exact sequence
\begin{equation} \label{eqn:450}
A_0 \ot^{}_{A} P^{-1} \xar{\d_0} A_0 \ot^{}_{A} P^0 \to L_0 \to 0 .
\end{equation}

Let $\bar{p}_1, \ldots, \bar{p}_r$ be elements of
$A_0 \ot_{A} P^0$ whose cohomology classes
$[\bar{p}_1], \ldots, [\bar{p}_r]$ generated $L_0$ as an $A_0$-module.
Let
$\phi_0 : A_0^{\oplus r} \to A_0 \ot_{A} P^0$
be the homomorphism corresponding to the sequence
$(\bar{p}_1, \ldots, \bar{p}_r)$.
Then the homomorphism
\begin{equation} \label{eqn:171}
\d_0 \oplus \phi_0 :
(A_0 \ot_{A} P^{-1}) \oplus A_0^{\oplus r} \to A_0 \ot_{A} P^{0}
\end{equation}
is surjective.

Choose elements $p_1, \ldots, p_r$ in $P^0$ lifting the elements
$\bar{p}_1, \ldots, \bar{p}_r$,
and let $\phi : \what{A}^{\oplus r} \to P^0$ be
the corresponding $\what{A}$-module homomorphism.
Define
$\psi := \d \oplus \phi : P^{-1} \oplus \what{A}^{\oplus r} \to P^0$.
Consider this commutative diagram of $A$-modules
\[ \begin{tikzcd} [column sep = 12ex, row sep = 6ex]
P^{-1} \oplus \what{A}^{\oplus r}
\ar[d, two heads]
\ar[r, "{\psi}"]
\ar[dr, "{\psi_0}", two heads]
&
P^0
\ar[d, two heads]
\\
(A_0 \ot_{A} P^{-1}) \oplus A_0^{\oplus r}
\ar[r,  two heads, "{\d_0 \oplus \phi_0}"]
&
A_0 \ot_{A} P^{0}
\end{tikzcd} \]
in which the vertical arrows are the canonical surjections. The homomorphism
$\d_0 \oplus \phi_0$ is known to be surjective, by (\ref{eqn:171}).
Hence $\psi_0$ is surjective. The $A$-modules
$P^{-1} \oplus \what{A}^{\oplus r}$ and $P^0$ are
$\a$-adically complete.
By the Complete Nakayama Theorem \ref{thm:275} the
homomorphism $\psi$ is surjective.
Comparing $\psi$ to (\ref{eqn:447}), we conclude that the cohomology
classes $[p_1], \ldots, [p_r]$ generate $L$ as an $\what{A}$-module.
\end{proof}

Here is Corollary \ref{cor:400} from the introduction.

\begin{cor} \label{cor:412}
In the setting of Theorem \ref{thm:104},
let $M$ and $N$ be derived $\a$-adically complete
complexes of $A$-modules, with
$\opn{sup}(\opn{H}(M)), \opn{sup}(\opn{H}(N)) \leq i_0$ for some
$i_0 \in \Z$.
Let $\phi : M \to  N$ be a morphism
in $\cat{D}(A)$.
The following two conditions are equivalent:
\begin{itemize}
\rmitem{i} The homomorphism
$\opn{H}^{i_0}(\phi) : \opn{H}^{i_0}(M) \to \opn{H}^{i_0}(N)$ is surjective.

\rmitem{ii} The homomorphism
\[ \opn{H}^{i_0}(\opn{id}_{A_0} \ot^{\mrm{L}}_{A} \msp{3} \phi) :
\opn{H}^{i_0}(A_0 \ot^{\mrm{L}}_{A} M) \to
\opn{H}^{i_0}(A_0 \ot^{\mrm{L}}_{A} N) \]
is surjective.
\end{itemize}
\end{cor}

\begin{proof}
Let $L$ be the standard cone of $\phi$. By \cite[Proposition 7.3.5]{Ye5} there
is a distinguished triangle
$M \xar{\phi} N \xar{} L \xar{\msp{3} \triangle \msp{3}}$
in $\cat{D}(A)$. Therefore $L \in \cat{D}(A)_{\tup{$\a$-com}}$,
and also $\opn{sup}(\opn{H}(L)) \leq i_0$.
There is an induced distinguished triangle
\[ A_0 \ot^{\mrm{L}}_{A} M \xar{\phi_0} A_0 \ot^{\mrm{L}}_{A} N \xar{}
A_0 \ot^{\mrm{L}}_{A} L \xar{\msp{3} \triangle \msp{3}} \]
in $\cat{D}(A_0)$, where
$\phi_0 := \opn{id}_{A_0} \ot^{\mrm{L}}_{A}  \msp{3} \phi$.

Condition (i) is equivalent to the condition $\opn{H}^{i_0}(L) = 0$, which
we call (i$'$).
Condition (ii) is equivalent to the condition
$\opn{H}^{i_0}(A_0 \ot^{\mrm{L}}_{A} L) = 0$, which
we call (ii$'$).
If $\opn{sup}(\opn{H}(L)) < i_0$ then condition (i$'$) holds trivially;
and if $\opn{sup}(\opn{H}(L)) = i_0$, the implication
(ii$'$) $\Rightarrow$ (i$'$) is a special case ($r = 0$) of
Theorem \ref{thm:104} above.
Since $\opn{sup}(\opn{H}(L)) \leq i_0$,
the derived K\"unneth trick \cite[Theorem 7]{Ye7} tells us that
$\opn{H}^{i_0}(A_0 \ot^{\mrm{L}}_{A} L) \cong
A_0 \ot_{A} \opn{H}^{i_0}(L)$,
and thus the implication (i$'$) $\Rightarrow$ (ii$'$) holds.
\end{proof}

\section{Completion Preserves Weak Proregularity}
\label{sec:WPR-and-comp}

In this section we prove Theorem \ref{thm:230} from the
Introduction, repeated here as Theorem \ref{thm:240}.
Convention \ref{conv:280} is in force.
Thus $\a \sub A$ is a finitely generated ideal, $\what{A} = \La_{\a}(A)$,
and $\what{\a} = \what{A} \cd \a \sub \what{A}$.

We already know that if $A$ is noetherian, then the ideal $\a$ is WPR.
Recall that one of the most fundamental facts about completion of noetherian
rings is this:
if $A$ is noetherian, then $\what{A}$ is noetherian too; implying that the
ideal $\what{\a}$ is WPR.
Another fundamental fact in the noetherian setting is that $\what{A}$ is flat
over $A$; this, combined with the next easy lemma, could also be used to
deduce the WPR of $\what{\a}$. Yet in the WPR setting (without the noetherian
property), completion can fail to be flat; see counterexample in
\cite[Theorem 7.2]{Ye4}.

These observations are an indication of the importance of Theorem \ref{thm:240}.

\begin{lem} \label{lem:510}
Let $A \to B$ be a flat ring homomorphism, let $\a \sub A$ be a WPR ideal, and
let $\b := B \cd \a \sub B$. Then the ideal $\b$ is WPR.
\end{lem}

\begin{proof}
Let $\bsym{a} = (a_1, \ldots, a_n)$ by a WPR generating sequence of $\a$, and
let $\bsym{b} = (b_1, \ldots, b_n)$ be the image of $\bsym{a}$ in $B$.
For every $j$ there is a canonical isomorphism of complexes
$\opn{K}(B; \bb^{j}) \cong B \ot_A \opn{K}(A; \ba^{j})$.
The flatness of $B$ gives a canonical isomorphism of modules
\[ B \ot_A \opn{H}^i \bigl( \opn{K}(A; \ba^{j}) \bigr) \cong
\opn{H}^i \bigl( B \ot_A \opn{K}(A; \ba^{j}) \bigr) \cong
\opn{H}^i \bigl( \opn{K}(B; \bb^{j}) \bigr) \]
for every $i$. So the inverse system
$\bigl\{ \opn{H}^i \bigl( \opn{K}(B; \bb^{j}) \bigr) \bigr\}_{j \in \N}$
is pro-zero for $i < 0$.
\end{proof}

Suppose $\b$ is some finitely generated ideal of $A$, such that
$\sqrt{\a} = \sqrt{\b}$. Then the sequences of ideals
$\{ \a^j \}_{j \geq 1}$ and $\{ \b^j \}_{j \geq 1}$ are cofinal.
This implies that the completion functors $\La_{\a}$ and $\La_{\b}$ are
isomorphic, and also the torsion functors $\Ga_{\a}$ and $\Ga_{\b}$ are
isomorphic. In particular, the ring $A / \b$ is an $\a$-torsion $A$-module.

\begin{lem} \label{lem:165}
Let $\b$ be a finitely generated ideal of $A$, such that
$\sqrt{\a} = \sqrt{\b}$.
Let $\what{\b} := \what{A} \cd \b$, the ideal in $\what{A}$ generated by $\b$.
Then the $A$-ring homomorphism
$A / \b \to \what{A} \msp{2} / \msp{2} \what{\b}$ is an isomorphism.
\end{lem}

\begin{proof}
The ring $\what{A}$ is also the $\b$-adic completion of $A$.
According to \cite[Theorem 2.8]{Ye4} or
\cite[Corollary 3.6 and Theorem 1.2]{Ye1},
the homomorphism $A / \b \to \what{A} \msp{2} / \msp{2} \what{\b}$ is an
isomorphism.
\end{proof}

The next lemma is the key technical result of this section.

\begin{lem} \label{lem:156}
Assume $\a$ is a WPR ideal.
Let $\bb = (b_1, \ldots, b_n)$ be a sequence of elements in $A$,
let $\b$ be the ideal generated by $\bb$, and assume $\sqrt{\a} = \sqrt{\b}$.
For every $i$ let $\what{b}_i$ be the image of the element $b_i$ in $\what{A}$,
and define the sequence
$\what{\bb} := (\what{b}_1, \ldots, \what{b}_n)$ in $\what{A}$.
Then the obvious homomorphism of Koszul complexes
$\opn{K}(A; \bb) \to \opn{K}(\what{A}; \what{\bb})$ is a quasi-isomorphism.
\end{lem}

\begin{proof}
Write $K := \opn{K}(A; \bb)$ and $\what{K} := \opn{K}(\what{A}; \what{\bb})$,
and let $\phi : K \to \what{K}$ be the obvious homomorphism of complexes of
$A$-modules, which is also a homomorphism of commutative DG $A$-rings.
We want to prove that $\phi$ is a quasi-isomorphism in
$\cat{C}_{\mrm{str}}(A)$,
and this is the same as proving that $\opn{Q}(\phi) : K \to \what{K}$
is an isomorphism in $\cat{D}(A)$.

Since $\opn{H}^0(K) \cong A / \b$ as $A$-rings,
all the cohomologies $\opn{H}^i(K)$ are $(A / \b)$-modules, and
therefore they are $\a$-torsion. According to
\cite[Corollary 4.32]{PSY1}
the complex $K$ belongs to $\cat{D}(A)_{\tup{$\a$-tor}}$.

By Lemma \ref{lem:165} the $A$-ring homomorphism
$A / \b \to \what{A} \msp{2} / \msp{2} \what{\b}$ is bijective.
Therefore $\opn{H}^0(\what{K}) \cong A / \b$ as $A$-rings,
all the cohomologies $\opn{H}^i(\what{K})$ are $(A / \b)$-modules, and so
the complex $\what{K}$ also belongs to
$\cat{D}(A)_{\tup{$\a$-tor}}$.

We need to prove that  the morphism
$\opn{Q}(\phi) : K \to \what{K}$ in $\cat{D}(A)$ is an
isomorphism.
Since both $K$ and $\what{K}$ belong to $\cat{D}(A)_{\tup{$\a$-tor}}$,
the MGM Equivalence \cite[Theorem 1.1]{PSY1} says that
$\opn{Q}(\phi) : K \to \what{K}$ is an
isomorphism in $\cat{D}(A)$ if and only if
\begin{equation} \label{eqn:165}
\mrm{L} \La_{\a}(\opn{Q}(\phi)) :
\mrm{L} \La_{\a}(K) \to \mrm{L} \La_{\a}(\what{K})
\end{equation}
is an isomorphism in $\cat{D}(A)$.
We are going to prove that $\mrm{L} \La_{\a}(\opn{Q}(\phi))$
is an isomorphism.

Observe that as a complex of $A$-modules,
$\what{K} = \La_{\a}(K)$, the $\a$-adic completion of the complex $K$.
The homomorphisms $\phi, \tau_K : K \to \what{K}$
in $\cat{C}_{\mrm{str}}(A)$ are equal,
and $\La_{\a}(\phi) : \La_{\a}(K) \to \La_{\a}(\what{K})$
is an isomorphism.

The complexes $K$ and $\what{K}$ are bounded complexes of $\a$-adically flat
$A$-modules. For $K$ this is clear, and for $\what{K}$ we use
\cite[Theorem 5.3]{Ye4}.
Consider the following commutative diagram in $\cat{D}(A)$.
\[ \begin{tikzcd} [column sep = 12ex, row sep = 6ex]
\mrm{L} \La_{\a}(K)
\ar[r, "{\mrm{L} \La_{\a}(\opn{Q}(\phi))}"]
\ar[d, "{\eta^{\mrm{L}}_{K}}"']
&
\mrm{L} \La_{\a}(\what{K})
\ar[d, "{\eta^{\mrm{L}}_{\what{K}}}"]
\\
\La_{\a}(K)
\ar[r, "{\opn{Q}(\La_{\a}(\phi))}"]
&
\La_{\a}(\what{K})
\end{tikzcd} \]
By Theorem \ref{thm:115} the vertical arrows are isomorphisms.
Since $\La_{\a}(\phi)$ is an isomorphism, so is
$\opn{Q}(\La_{\a}(\phi))$. We conclude that
$\mrm{L} \La_{\a}(\opn{Q}(\phi))$ is an isomorphism.
\end{proof}

\begin{thm} \label{thm:240}
Let $A$ be a ring, let $\a$ be a WPR ideal in $A$,
let $\what{A}$ be the $\a$-adic completion of $A$, and let
$\what{\a} := \what{A} \cd \a$, the ideal in $\what{A}$ generated by $\a$. Then
the ideal $\what{\a} \sub \what{A}$ is weakly proregular.
\end{thm}

As mentioned at the end of the Introduction, Theorem \ref{thm:240} was
already known to Positselski, see \cite[Example 5.2(2)]{Po1}.
We discovered this theorem independently, unaware of the prior work of
Positselski. As far as we can tell, our proof
(based on the MGM Equivalence and Lemma \ref{lem:156}) is totally different
from the proof outlined by Positselski.

\begin{proof}
Choose some WPR sequence
$\ba = (a_1, \ldots, a_n)$ of elements of $A$ that generates the ideal $\a$.
For every $j \geq 1$ define the sequence
$\ba^j := (a_1^j, \ldots, a_n^j)$,
and let $\b_j$ be the ideal of $A$ generated by the sequence $\ba^j$.
This ideal satisfies $\sqrt{\b_j} = \sqrt{\a}$.

For every $i$ let $\what{a}_i$ be the image of the element $a_i$ in
$\what{A}$. Define the sequences
$\what{\ba} := (\what{a}_1, \ldots, \what{a}_n)$
and
$\what{\ba}^j := (\what{a}_1^{\msp{2} j}, \ldots, \what{a}_n^{\msp{2} j})$ in
the ring $\what{A}$.
The sequence $\what{\ba}$ generates the ideal $\what{\a}$,
and the sequence $\what{\ba}^{\msp{2} j}$ generates the ideal
$\what{\b}_j := \what{A} \cd \b_j \sub \what{A}$.
Again,
$\sqrt{\msp{3} \what{\b}_j \msp{2}} =
\sqrt{\msp{2} \what{\a} \msp{2}}$.

According to Lemma \ref{lem:156}, for every $j$ the homomorphism of Koszul
complexes
$\phi_j : \opn{K}(A; \ba^j) \to \opn{K}(\what{A}; \what{\ba}^j)$
is a quasi-isomorphism. Therefore, for every $i$ and $j$, the $A$-module
homomorphism
\[ \opn{H}^i(\phi_j) :
\opn{H}^i \bigl( \opn{K}(A; \ba^j) \bigr) \to
\opn{H}^i \bigl( \opn{K}(\what{A}; \what{\ba}^j) \bigr) \]
is an isomorphism. So for every $i$ there is an isomorphism
\begin{equation} \label{eqn:167}
\bigl\{ \opn{H}^i(\phi_j) \bigr\}_{j \geq 1} :
\bigl\{ \opn{H}^i \bigl( \opn{K}(A; \ba^j) \bigr) \bigr\}_{j \geq 1}
\iso
\bigl\{ \opn{H}^i \bigl( \opn{K}(\what{A}; \what{\ba}^j) \bigr)
\bigr\}_{j \geq 1}
\end{equation}
of inverse systems of $A$-modules.

We are given that $\ba$ is a WPR sequence in $A$.
This means that for every $i < 0$ the first inverse system of $A$-modules
in (\ref{eqn:167}) is pro-zero. But then the second inverse system there is
also pro-zero, and thus $\what{\ba}$ is a WPR sequence in $\what{A}$.
Since the sequence $\what{\ba}$ generates the ideal $\what{\a}$, it follows
that this ideal of $\what{A}$ is WPR.
\end{proof}

The ring homomorphism $A \to \what{A}$ induces the restriction functors
$\opn{Rest} : \cat{M}(\what{A}) \to \cat{M}(A)$,
$\opn{Rest} : \cat{K}(\what{A}) \to \cat{K}(A)$ and
$\opn{Rest} : \cat{D}(\what{A}) \to \cat{D}(A)$.

\begin{prop} \label{prop:453}
The functor
$\opn{Rest} : \cat{K}(\what{A}) \to \cat{K}(A)$
induces an isomorphism of categories
$\opn{Rest} : \cat{K}(\what{A})_{\msp{2} \tup{$\what{\a}$-kpr}} \to
\cat{K}(A)_{\tup{$\a$-kpr}}$.
\end{prop}

\begin{proof}
An $\a$-adically complete $A$-module is the same as an
$\what{\a}$-adically complete $\what{A}$-module.
So let us refer to them as complete modules.
According to Proposition \ref{prop:440}(5), for complete modules $M$ and $N$
there is equality
$\opn{Hom}_{\what{A}}(M, N) = \opn{Hom}_{A}(M, N)$.
Therefore the full DG subcategory of $\cat{C}(A)$ on the
complexes of complete modules is equal to the full DG subcategory of
$\cat{C}(\what{A})$ on the complexes of complete modules.
This implies that the full triangulated subcategory of $\cat{K}(A)$ on the
complexes of complete modules is equal to the full triangulated subcategory
of $\cat{K}(\what{A})$ on the complexes of complete modules.

Take complexes of complete modules $P$ and $N$, with $N$ acyclic.
Since
$\opn{Hom}_{\what{A}}(P, N) = \opn{Hom}_{A}(P, N)$,
we see that $P$ is $\a$-adically K-projective over $A$ if and only if
$P$ is $\what{\a}$-adically K-projective over $\what{A}$.
Conclusion:
$\cat{K}(\what{A})_{\msp{2} \tup{$\what{\a}$-kpr}} =
\cat{K}(A)_{\tup{$\a$-kpr}}$.
\end{proof}

\begin{thm} \label{thm:235}
Let $A$ be a ring, let $\a$ we a WPR ideal in $A$,
let $\what{A}$ be the $\a$-adic completion of $A$, and let
$\what{\a} := \what{A} \cd \a$, the ideal in $\what{A}$ generated by $\a$. Then:
\begin{enumerate}
\item A complex of $\what{A}$-modules $M$ is derived $\what{\a}$-adically
complete if and only if $\opn{Rest}(M)$ is a derived $\a$-adically
complete complex of $A$-modules.

\item The restriction functor
\[ \opn{Rest} : \cat{D}(\what{A} \msp{2})_{\msp{2} \tup{$\what{\a}$-com}} \to
\cat{D}(A)_{\tup{$\a$-com}} \]
is an equivalence of triangulated categories.
\end{enumerate}
\end{thm}

\begin{proof}
(1) To simplify the presentation, and because we are going to work with explicit
quasi-isomorphisms in $\cat{C}_{\mrm{str}}(A)$ and
$\cat{C}_{\mrm{str}}(\what{A})$, we shall write $M$ instead of
$\opn{Rest}(M)$.

Let $\ba = (a_1, \ldots, a_p)$ be a finite sequence of elements in $A$ that
generates $\a$. We know that the sequence $\ba$ is WPR.
Let $T := \opn{Tel}(A; \ba)$, the telescope complex associated to $\ba$,
and let
$\tau^{\msp{2} \ba}_{M} : M \to \opn{Hom}_{A}(T, M)$
be the canonical homomorphism in $\cat{C}_{\mrm{str}}(A)$, as recalled in
Section \ref{sec:back-der}.
According to \cite[Corollary 5.25]{PSY1}, $M$ is a derived $\a$-adically
complete complex of $A$-modules if and only if
$\tau^{\msp{2} \ba}_{M}$ is a quasi-isomorphism.

Let $\what{\ba}$ be the image of the sequence $\ba$ in the ring $\what{A}$.
This sequence generates the ideal $\what{\a}$. By Theorem \ref{thm:240} we know
that $\what{\a}$ is a WPR ideal. Hence $\what{\ba}$ is a WPR sequence in
$\what{A}$. Let $\what{T} := \opn{Tel}(\what{A}; \what{\ba})$. Then, as above,
$M$ is a derived $\what{\a}$-adically complete complex of
$\what{A}$-modules if and only if
the homomorphism
$\tau^{\msp{2} \what{\ba}}_{M} : M \to
\opn{Hom}_{\what{A}}(\what{T}, M)$
in $\cat{C}_{\mrm{str}}(\what{A})$
is a quasi-isomorphism.

Now the obvious homomorphism $\what{A} \ot_A T \to \what{T}$
is an isomorphism in
$\cat{C}_{\mrm{str}}(\what{A})$. Hence
\begin{equation} \label{eqn:280}
\opn{Hom}_{\what{A}}(\what{T}, M) \cong
\opn{Hom}_{\what{A}}( \what{A} \ot_A T, M)
\cong \opn{Hom}_{A}(T, M)
\end{equation}
in $\cat{C}_{\mrm{str}}(A)$. Checking the definitions (see Section
\ref{sec:back-der})  we see that the diagram
\[ \begin{tikzcd} [column sep = 8ex, row sep = 6ex]
M
\ar[d, "{\tau^{\ba}_{M}}"']
\ar[dr, "{\tau^{\msp{2} \what{\ba}}_{M}}"]
\\
\opn{Hom}_{A}(T, M)
\ar[r, "{\simeq}"]
&
\opn{Hom}_{\what{A}}(\what{T}, M)
\end{tikzcd} \]
in $\cat{C}_{\mrm{str}}(A)$, in which the horizontal isomorphism is
(\ref{eqn:280}),  is commutative.
So $\tau^{\msp{2} \ba}_{M}$ is a quasi-isomorphism if and only if
$\tau^{\msp{2} \what{\ba}}_{M}$ is a quasi-isomorphism.

\medskip \noindent
(2) Consider the diagram of functors
\[ \begin{tikzcd} [column sep = 8ex, row sep = 6ex]
\cat{K}(\what{A} \msp{2})_{\tup{$\what{\a}$-kpr}}
\ar[d, "{\opn{Rest}}"']
\ar[r, "{\mrm{Q}}"]
&
\cat{D}(\what{A} \msp{2})_{\tup{$\what{\a}$-com}}
\ar[d, "{\opn{Rest}}"]
\\
\cat{K}(A \msp{2})_{\tup{$\a$-kpr}}
\ar[r, "{\mrm{Q}}"]
&
\cat{D}(A)_{\tup{$\a$-com}}
\end{tikzcd} \]
which is commutative up to a canonical isomorphism.
According to Proposition \ref{prop:430} and Theorem \ref{thm:135}, the
horizontal localization functors are equivalences. By Proposition
\ref{prop:453}, the restriction functor on the left is an isomorphism of
categories. Hence the restriction functor on the right is an equivalence of
categories.
\end{proof}

We end this section with two results of a practical nature.
{\em Adic flatness} was introduced in Definition \ref{dfn:260}; recall
that other texts, including \cite{BS}, use the adjective {\em complete
flatness}.

\begin{thm} \label{thm:510}
Let $A \to B$ be a flat ring homomorphism, and let $M$ be a flat $B$-module.
Let $\a \sub A$ be a weakly proregular ideal, and define the ideal
$\b := B \cd \a \sub B$. Let $\wh{B}$ be the $\b$-adic completion of $B$, with
ideal $\wh{\b} := \wh{B} \cd \b \sub B$, and let
$\wh{M}$ be the $\b$-adic completion of $M$.
Then $\wh{M}$ is a $\wh{\b}$-adically flat $\wh{B}$-module.
\end{thm}

\begin{proof}
By Lemma \ref{lem:510} the ideal $\b \sub B$ is WPR. By Theorem \ref{thm:240}
the ideal $\wh{\b} \sub \wh{B}$ is WPR.
For every $i \geq 0$ let
$B_i := B / \b^{i + 1}$ and $M_i := B_i \ot_B M$; so
$\{ M_i \}_{i \geq 0}$ is a flat $\b$-adic system of $B$-modules.
Now according to Proposition \ref{prop:440}(1) the canonical homomorphism
$B_i \to B_i \ot_B \wh{B} \cong \wh{B} / \wh{\b}^{i + 1}$ is bijective.
We see that $\{ M_i \}_{i \geq 0}$ is a flat $\wh{\b}$-adic system of
$\wh{B}$-modules.
Lastly, according to \cite[Theorem 6.9]{Ye4} the $\wh{B}$-module
$\wh{M} = \lim_{\lto i} \msp{1} M_i$ is
a $\wh{\b}$-adically flat $\wh{B}$-module.
\end{proof}

\begin{rem} \label{rem:510}
We do not know whether Theorem \ref{thm:510} can be made stronger  by
asserting that $\wh{M}$ is a {\em flat} $\wh{B}$-module.
If the ring $\wh{B}$ happens to be noetherian, then this is true, by
\cite[Theorem 1.5]{Ye4}. Hence this rules out
\cite[Theorem 7.2]{Ye4} from being a counterexample.
\end{rem}

\begin{cor} \label{cor:510}
Let $A \to B \to C$ be flat ring homomorphisms, with $A$ noetherian.
Given an ideal $\a \sub A$, let $\b := B \cd \a \sub B$ and
$\c := C \cd \a \sub C$ be the induced ideals, and let
$\wh{B}$ and $\wh{C}$ be the corresponding completions of $B$ and $C$.
Define the ideal $\wh{\b} := \wh{B} \cd \b \sub \wh{B}$.
Then $\wh{C}$ is $\wh{\b}$-adically flat over $\wh{B}$.
\end{cor}

\begin{proof}
The ideal $\a$ is WPR because $A$ is noetherian.
The ring $\wh{C}$ is also the $\b$-adic completion of $C$.
Now use the theorem, with $M := C$.
\end{proof}


\end{document}